\documentclass[12pt]{amsart}
\usepackage{fullpage,url,amssymb,enumerate,colonequals}

\usepackage{mathrsfs} 
\usepackage{MnSymbol}
\usepackage{extarrows}
\usepackage{lscape}
\usepackage[all,cmtip]{xy}

\usepackage[OT2,T1]{fontenc}

\usepackage{enumerate}

\usepackage{xcolor}
\usepackage{soul}

\usepackage{cancel} 

\usepackage[
        colorlinks, citecolor=darkgreen,
        backref,
        pdfauthor={Nuno Freitas}, 
]{hyperref}

\usepackage{comment}

\newcommand{\F}{\mathbb{F}}

\newcommand{\Q}{\mathbb{Q}}

\newcommand{\Z}{\mathbb{Z}}
\newcommand{\Qbar}{{\overline{\Q}}}

\newcommand{\rhobar}{{\overline{\rho}}}

\newcommand{\calF}{\mathcal{F}}

\newcommand{\calL}{\mathcal{L}}

\newcommand{\calO}{\mathcal{O}}

\newcommand{\fp}{\mathfrak{p}}
\newcommand{\fq}{\mathfrak{q}}

\DeclareMathOperator{\Frob}{Frob}
\DeclareMathOperator{\Gal}{Gal}

\DeclareMathOperator{\Norm}{Norm}

\newcommand{\tors}{{\operatorname{tors}}}

\newcommand{\unr}{{\operatorname{unr}}}

\newcommand{\vv}{\upsilon}

\usepackage[normalem]{ ulem }
\usepackage{soul}

\numberwithin{equation}{section}

\newtheorem{theorem}{Theorem}
\newtheorem{lemma}{Lemma}

\newtheorem{proposition}{Proposition}

\theoremstyle{definition}

\theoremstyle{remark}
\newtheorem{remark}[equation]{Remark}

\definecolor{darkgreen}{rgb}{0,0.5,0}

\begin{document}

\title{A multi-Frey approach to Fermat equations of signature $(r,r,p)$}

\author{Nicolas Billerey}
\address{Universit\'e Clermont Auvergne, CNRS, LMBP, F-63000 Clermont-Ferrand, France.
}
\email{Nicolas.Billerey@uca.fr}

\author{Imin Chen}

\address{Department of Mathematics, Simon Fraser University\\
Burnaby, BC V5A 1S6, Canada } \email{ichen@sfu.ca}

\author{Luis Dieulefait}

\address{Departament d'Algebra i Geometria,
Universitat de Barcelona,
G.V. de les Corts Catalanes 585,
08007 Barcelona, Spain}
\email{ldieulefait@ub.edu}

\author{Nuno Freitas}
\address{
Instituto de Ciencias Matem\'aticas, CSIC,
Calle Nicol\'as Cabrera
13--15, 28049 Madrid, Spain}
\email{nuno.freitas@icmat.es}

\date{\today}


\thanks{N.B. acknowledges the financial support of ANR-14-CE-25-0015 Gardio. The last author was partly supported by the grant {\it Proyecto RSME-FBBVA $2015$ Jos\'e Luis Rubio de Francia}.}

\begin{abstract}
In this paper, we give a resolution of the generalized Fermat equations $$x^5 + y^5 = 3 z^n \text{ and }  x^{13} + y^{13} = 3 z^n,$$ for all integers $n \ge 2$,  and all integers $n \ge 2$ which are not a power of $7$, respectively, using the modular method with Frey elliptic curves over totally real fields. The results require a refined application of the multi-Frey technique, which we show to be effective in new ways to reduce the bounds on the exponents $n$.

We also give a number of results for the equations $x^5 + y^5 = d z^n$, where $d = 1, 2$, under additional local conditions on the solutions. This includes a result which is reminiscent of the second case of Fermat's Last Theorem, and which uses a new application of level raising at $p$ modulo $p$.
\end{abstract}

\maketitle


\section{Introduction}

Wiles' 1995 proof \cite{Wiles} of Fermat's Last Theorem
pioneered a new strategy to attack Diophantine equations, 
now known as {\it the modular method}. 
The strategy, originally due to Frey, Serre, Ribet and Wiles is to attach to a putative solution
of a Diophantine equation an elliptic curve $E$ (known as a Frey elliptic curve), and study the mod $p$
representation attached to $E$ via modularity and level lowering. 
This relates the solution to a modular form of weight $2$ and small level and, to conclude, one needs to show that such relation leads to a contradiction (see Section~\ref{S:Overview} for more details).

The idea of using this same strategy to study variants of FLT goes back to the  work of Serre \cite[Section~4.3]{Serre87} and Ribet \cite{RibetFermat}. Since Wiles' breakthrough, mathematicians have generalized and improved the method and applied it to many other Diophantine equations. In particular, it was natural to use the modular approach to study the \emph{Generalized Fermat Equation}
\begin{equation}
 Ax^p+By^q=Cz^r, \qquad  p,q,r \in \mathbb{Z}_{\geq 2}, \qquad A,B,C \in \Z_{\ne 0}
\label{E:GFE}
\end{equation}
with $A,B,C$ pairwise coprime. This equation is subject of the following conjecture.

{\bf Conjecture.} Fix $A,B,C$ as above.
Over all choices of prime exponents $p,q,r$ satisfying $1/p+1/q+1/r<1$ the 
equation \eqref{E:GFE} admits only finitely many solutions $(a,b,c)$ such that
$abc \ne 0$ and $\gcd(a,b,c)=1$. (Here solutions like $2^3 + 1^q = 3^2$ are counted only once.) 

The only general result towards the above conjecture is a theorem due to Darmon and Granville \cite{DG} which states that if besides $A,B,C$ we also fix the prime exponents $p,q,r$ then there are only finitely many solutions as above. The conjecture is also known to hold in some particular cases including certain infinite families, for which  the authors of this paper have previously made contributions.  Moreover, it is also known that the full conjecture is a consequence of the $ABC$-conjecture (see \cite[Section~5.2]{DG}).

Bennett \cite{BenSki}, \cite{BenVatYaz}, Kraus \cite{kraus2}, \cite{kraus1}  and Siksek  \cite{BuMiSi2}, \cite{BuMiSi1} and their collaborators have developed and clarified the method using Frey elliptic curves over $\Q$. Unfortunately, there is a restrictive set of exponents $(p,q,r)$ which can be approached using the modular method over $\Q$ due to constraints coming from the classification of Frey representations \cite{DG}. As a consequence, attention has now shifted towards using Frey elliptic curves over totally real fields, and is made possible because of advances on the Galois representation side (i.e. modularity results).

In this paper, we establish further cases of the conjecture above based on 
extensions of the modular method to the setting of Hilbert modular forms as 
introduced in the work of the last two authors \cite{DF2}, and powered by the multi-Frey
technique as explained by Siksek in \cite{BuMiSi3}, \cite{BuMiSi4}.  

The results in this paper provide evidence that the multi-Frey technique applied with a `sufficiently rich' set of Frey curves can be used to `patch together' a complete resolution of a one parameter family of generalized Fermat equations. As it will be seen throughout the paper, the multi-Frey technique complements methods used in several steps in the modular method, allowing for refined bounds.

\subsection{Our Diophantine results}

Let $d\geq 1$ be an integer. We are concerned with 
Fermat type equations of the form 
\begin{equation}
  x^r + y^r = dz^p, \qquad xyz \ne 0, \qquad \gcd(x,y,z) = 1
  \label{E:rrp}
\end{equation}
where $r, \; p$ are prime exponents with $r$ fixed and $p$ is allowed to vary.

We say that a solution $(x,y,z)=(a,b,c)$ of equation  \eqref{E:rrp} is {\it
non-trivial} if it satisfies $|abc| > 1$ and we call it {\it primitive} if 
$\gcd(a,b,c) = 1$. In the case
of most interest to us, $d=3$, the condition $|abc| > 1$ is equivalent to
$abc \neq 0$, but it is important to note that for $d=2$ there are also the
extra trivial solutions $\pm (1,1,1)$.

The equation~\eqref{E:rrp} with $r = 5$ and $d = 2, 3$  has already been subject of the papers \cite{Bill}, \cite{BD} and \cite{DF1}, where it was resolved for $3/4$ of prime exponents $p$. For $r= 13$ and $d = 3$, it has been resolved in the papers \cite{DF2}, \cite{FS} under the assumption $13 \nmid z$.

Our main Diophantine results are that we completely solve equation \eqref{E:rrp} for $d = 3$ when $r = 5$ (resp.\ $r = 13$) and $p = n \geq 2$ is any integer (resp.\ $p = n \geq 2$ is any integer which is not a power of $7$). Clearly, this will follow directly from the same statements for prime exponents. More precisely, we will prove the following theorems.
\begin{theorem}\label{T:main5}
For all primes $p$, there are no non-trivial
primitive solutions to
\begin{equation}
\label{E:55p}
    x^5 + y^5 = 3 z^p.
\end{equation} 
\end{theorem}

\begin{theorem} 
For all primes $p \not= 7$, there are no non-trivial primitive solutions to
\begin{equation}
\label{E:1313p}
  x^{13} + y^{13} = 3 z^p.
\end{equation}
\label{T:main13}
\end{theorem}
In the previous papers concerning equations~\eqref{E:55p} and \eqref{E:1313p}, the main tool used was the modular method, where the Frey elliptic curves were obtained by exploiting the factorization over $\Q(\zeta_r)$ (for $r=5$ or $r=13$) of the left-hand side of each equation. More generally, in the work of the last author \cite{F}, for each $r \geq 5$, several Frey elliptic curves defined over real subfields of $\Q(\zeta_r)$ are attached to equation \eqref{E:rrp}. Our proofs of Theorems~\ref{T:main5} and \ref{T:main13} build on these previous works and are made possible by introducing new multi-Frey techniques.

In particular, we show how the multi-Frey technique can be used to obtain tight bounds on the exponent $p$, improve bounds coming from Mazur-type irreducibility results (see Theorem~\ref{T:irreducibilityF}), and move to another level where the required computations of Hilbert modular forms is within the range of what is currently feasible (see paragraph after Lemma~\ref{L:Fconductors}). We also need a refined `image of inertia argument' (see Section~\ref{S:inertia}) for the elimination step of the modular method.

A major obstruction to the success of the modular method for solving \eqref{E:rrp} for $d=1,2$ 
is the existence of trivial solutions like $(1,0,1)$, $(1,-1,0)$ or $(1,1,1)$. Indeed, when the Frey elliptic curve evaluated at a trivial solution is non-singular then its corresponding (via modularity) newform will be among the newforms after level lowering; in particular, the mod $p$ representations of the Frey curve and a newform can be isomorphic, requiring the use of global methods to distinguish Galois representations which are uniform in $p$.

It is sometimes possible to resolve equation $\eqref{E:rrp}$ by assuming additional $q$-adic conditions to avoid the obstructing trivial solutions. Indeed, we will prove a number of partial results for the equation \eqref{E:rrp} with $r = 5$ and $d=1,2$ under certain $q$-adic conditions.

For example, we will prove the following result resembling the second case of Fermat's Last Theorem. Its proof involves a new application of the condition for level raising at $p$ modulo~$p$.
\begin{theorem}
For all primes $p$, the equation 
\begin{equation*}
x^5 + y^5 = dz^p,\qquad\text{with $d\in\{1,2\}$}
\end{equation*}
has no non-trivial solutions $(a,b,c)$ satisfying~$p \mid c$. 
\label{T:second case}
\end{theorem}

In addition, we will use the multi-Frey technique to prove the following result, which was known in the case $d=1$ by work of Billerey (\cite[Th\'eor\`eme~1.1]{Bill}) and Dahmen-Siksek (\cite[Proposition~3.3]{DahSik}) using the Frey curve introduced in section~\ref{S:overQ}.
\begin{theorem}
For all primes $p$, the equation 
\begin{equation*}
x^5 + y^5 = dz^p 
\end{equation*}
has no non-trivial solutions $(a,b,c)$ in each of the following situations~:
\begin{itemize}
 \item[(i)] $d = 1,2$ and $5 \mid c$ or,
 \item[(ii)] $d=1$ and $c$ even or,
 \item[(iii)] $d=2$ and $c$ even.
\end{itemize}
\label{T:partial}
\end{theorem}
We remark that, in all our theorems, to deal with certain small primes, we invoke references where the results are obtained using Frey elliptic curves different from the ones used in this paper; this is another instance of the multi-Frey technique.

The computations required to support the proof of our main theorems were performed using {\tt Magma} \cite{magma}. The program files are provided with this paper and we refer to \cite{programs} whenever an assertion involves a computation in {\tt Magma} from one of these programs.

\subsection*{Acknowledgments}
We thank Christophe Breuil for discussions concerning Proposition~\ref{P:levelraising}. 
We also thank Peter Bruin for helpful conversations regarding Remark~\ref{R:example},
Maarten Derickx and Michael Stoll for a conversation about \cite{smallTorsion}, and Enrique Gonzalez Jimenez for helpful suggestions. We are grateful to the referee for a careful reading and helpful remarks.

\section{Overview of the multi-Frey modular method}

\label{S:Overview}

\noindent {\bf Notation:} Let $\Qbar$ be an algebraic closure of~$\Q$ and let $p$ be a prime number. 
For a totally real subfield $K$ of~$\Qbar$, we write $G_K = \Gal(\Qbar /K)$ for its absolute Galois group.
For a prime $\ell$ of $K$ we write $I_\ell$ for an inertia subgroup at $\ell$ in $G_K$.
Given $E$ an elliptic curve defined over~$K$, we denote by
$\rhobar_{E,p}$ the representation giving the action of~$\Gal(\Qbar/K)$ on the $p$-torsion points of~$E$.
For a Hilbert modular form $f$ defined over $K$ and a prime ideal $\fp$ in its field of coefficients~$\Q_f$, we write $\rhobar_{f,\fp}$ for the mod~$\fp$ Galois representation attached to $f$; when $K=\Q$ we get classical modular forms.

We now recall the main steps of the modular method.
\begin{itemize}
\item[]{\bf Step 1: Constructing a Frey curve.} Attach a Frey elliptic curve
$E/K$ to a putative solution of a Diophantine equation, where $K$ is a
totally real field. A Frey curve $E/K$ has the property that the Artin conductor of $\rhobar_{E,p}$ is bounded independently of the putative solution.
\item[]{\bf Step 2: Irreducibility.} Prove the irreducibility of 
$\overline{\rho}_{E,p}$.
\item[]{\bf Step 3: Modularity.} Prove the modularity of $E/K$, and hence modularity of $\rhobar_{E,p}$.
\item[]{\bf Step 4: Level lowering.} Use level lowering theorems, which require irreducibility of $\rhobar_{E,p}$, to conclude that $\overline{\rho}_{E,p} \cong \rhobar_{f,\fp}$ where $f$ is a Hilbert
newform over $K$ of parallel weight 2, trivial character, and level among
finitely many explicit possibilities $N_i$ and~$\fp$ is a prime ideal  above~$p$ in the field of coefficients~$\Q_f$ of~$f$.
\item[]{\bf Step 5: Contradiction.} Compute all the Hilbert
newforms predicted in Step 4 and show that $\overline{\rho}_{E,p} \not\cong
\rhobar_{f,\fp}$ for all of them. This typically uses various methods to distinguish local Galois representations.
\end{itemize}
In current applications of the modular method, the most
challenging step is often Step~$5$, contrasting with the proof
of Fermat's Last Theorem (the origin of the modular method) 
where the big issue was modularity. Indeed, in the proof of FLT
we have $K = \Q$ and in Step~$4$ there is only one level $N_1= 2$; 
since there are no newforms at this level 
we get directly a contradiction in Step~$5$. 
In essentially every other application of the method, there are candidates for $f$, therefore more work is needed to complete the argument, namely Step~$5$. 
It is now convenient for us to divide Step 5 into two substeps.
\begin{itemize}
 \item[]{\bf Step 5a: Computing newforms.} Compute all the Hilbert newforms of parallel weight 2, trivial character and levels $N_i$ predicted in Step~$4$.
 \item[]{\bf Step 5b: Discarding newforms.} For each newform $f$ computed in Step 5a and each prime ideal~$\fp$ above~$p$ in its field of coefficients show that $\overline{\rho}_{E,p} \not\cong
\rhobar_{f,\fp}$.
\end{itemize}
With the objective of succeeding more often in Step 5, Siksek introduced the {\bf multi-Frey technique} in \cite{BuMiSi3} and \cite{BuMiSi4}. This is a variant of Step~$1$ where more than one Frey curve is used simultaneously in order to put more restrictions on the putative solutions, thereby increasing the likelihood of a contradiction in Step~5b.

It is a common assumption in discussions about the modular method found in the literature
that Step 5a can be completed. We want to stress that more recently Step 5a is becoming a real obstruction to the method. This computational obstruction was not noticed in initial applications since they only required small (even empty) spaces of newforms over $\Q$ which were easily accessible. However, when working over totally real fields this is no longer the case as the dimensions of spaces of Hilbert cusp forms grow very fast. 

Besides the Diophantine results mentioned in the Introduction, one of the
underlying themes of this paper is to illustrate 
that the multi-Frey approach is a powerful and versatile tool 
with applications at various stages of the modular method.
Indeed, in the proofs of our main results, we will use it to circumvent 
challenges in Steps 2, 5a, and 5b.

{\bf Notation.} Let~$E$ be an elliptic curve defined over a totally real field~$K$ and let~$q$ be a rational prime such that~$E$ has good reduction at each prime ideal~$\fq$ dividing~$q$ in~$K$. For~$f$ a Hilbert newform over~$K$ of parallel weight~$2$ and trivial character, define
\[
B_q(E,f)=\gcd\left( \left\{ \Norm\left( a_{\fq}(E)-a_{\fq}(f)\right): \fq\mid q \right\} \right)
\]
where $\fq$ runs through the prime ideals above~$q$ in~$K$. Here~$a_{\fq}(f)$ denotes the $\fq$-th Fourier coefficient of~$f$ and \(a_{\fq}(E)=\#\F_{\fq}+1-\#\widetilde{E}(\F_{\fq})\) where $\F_{\fq}$ is the residual field at~$\fq$ and~$\widetilde{E}$ denotes the reduction of~$E$ modulo~$\fq$.

If for some prime ideal~$\fp$ above~$p$ in the coefficient field of~$f$ we have $\overline{\rho}_{E,p} \cong \rhobar_{f,\fp}$, then by considering the trace of Frobenius elements at each prime ideal above~$q$ on both sides, we get that $p$ divides~$q B_q(E,f)$. 

Throughout the paper, we write $\vv_{\fq}(a)$ for the valuation at the prime ideal~$\fq$ of the ideal generated by~$a\in K$.

\section{The image of inertia argument}
\label{S:inertia}

In this section we recall and generalize the `image of inertia argument'. 
This technique, originated in \cite{BenSki}, is used
to distinguish local Galois representations in Step 5b of the modular method. 
We start with the well known version,  and then provide two generalizations. 
All three versions are used later in the paper.

Let $L$ be a finite extension of $\Q_\ell$ contained in some fixed algebraic closure~$\Qbar_\ell$ of~$\Q_\ell$. Let
$E/L$ be an elliptic curve with potentially good reduction. Let $m \in
\Z_{\geq 3}$ be coprime to $\ell$ and consider the {\it inertial field of $E$}
given by $L_E = L^{un}(E[m])$, where $L^{un}$ is the maximal unramified extension of $L$ in~$\Qbar_\ell$.
The extension $L_E / L^{un}$ is independent of $m$ and it is the minimal extension of $L^{un}$ where $E$ achieves good reduction.

Suppose that, for a prime $p \not= \ell$, we have
\begin{equation}
  \overline{\rho}_{E,p} \cong \rhobar_{Z,p}, 
  \label{E:inertiaArg}
\end{equation}
where $E$ and $Z$ are elliptic curves over the local field $L$ and let $I_L$ denote the inertia subgroup of $L$.
In our applications below $E$ and $Z$ will be defined over a totally
real number field $K$ and $L$ will be the completion of $K$ at
some prime of $K$ above $\ell$. Moreover, $E$ will be a Frey elliptic curve and
$Z$ an elliptic curve corresponding to a (Hilbert) newform
with rational coefficients, as predicted in Step 4 of the modular method. 
The objective of the inertia argument is to obtain a contradiction to \eqref{E:inertiaArg},
thereby establishing
\begin{equation}
  \overline{\rho}_{E,p} \not\cong \rhobar_{Z,p}, 
  \label{E:inertiaArgNot}
\end{equation}
as required in Step 5b. We will now describe the three versions,
each version generalizing the previous one.

{\bf Version 1: different inertia sizes.} 
Show that $\# \rhobar_{E,p}(I_L) \ne \# \rhobar_{Z,p}(I_L)$; 
this clearly implies \eqref{E:inertiaArgNot}. This is effective when 
one curve has potentially good reduction and the other has 
potentially multiplicative reduction, and is the original version 
which has been used in many papers applying the modular method.

{\bf Version 2: the field of good reduction.} 
Suppose both $E$ and $Z$ have potentially good reduction.
Note that the inertial field $L_E$ corresponds to the field fixed
by the restriction $\rhobar_{E,p}|_{I_L}$ and that
isomorphism \eqref{E:inertiaArg}
implies $\overline{\rho}_{E,p}|_{I_L} \cong \rhobar_{Z, p}|_{I_L}$. Then
the inertial fields of $E$ and $Z$ must be the same. Therefore, even when
$\# \rhobar_{E,p}(I_L) = \# \rhobar_{Z,p}(I_L)$ (i.e.\ version 1 fails) we can still 
establish \eqref{E:inertiaArgNot} by showing that $L_E  \ne L_Z$ (working in a fixed algebraic closure of $L$). 

In practice, this is achieved by finding an extension $M/L$ 
where $Z$ has good reduction and $E$ does not. 
Indeed, consider the compositum $M' =L^{un} M$, which is an unramified extension of $M$.
Therefore, the type of reduction of $E$ and $Z$ over $M'$ is the same as over $M$. 
Since $Z/M'$ has good reduction by minimality of $L_Z$, it follows $L_Z \subset M'$; since $E/M'$ does not have good reduction, we have $L_E \not\subset M'$, and hence $L_E \neq L_Z$. We note that (when both curves have potentially good reduction) version 1 boils down to showing that $L_E$ and $L_Z$ 
are different because they have different degrees over $L^{un}$. This version was used in \cite{BeChDaYa} for instance.

{\bf Version 3: different conductors.} Let $M$ be an extension of $L$ 
and $G_M \subset \Gal(\overline{L}/ L)$ its corresponding subgroup.
Note that the isomorphism \eqref{E:inertiaArg} implies 
that $\rhobar_{E,p} |_{G_M} \cong \rhobar_{Z,p}|_{G_M}$. 
In particular, the restrictions $\rhobar_{E,p} |_{G_M}$
and $\rhobar_{Z,p}|_{G_M}$ must have the same conductor exponent.
Therefore, we can establish \eqref{E:inertiaArgNot}
if we find a field $M/L$ where the two restrictions 
have different conductor exponents. 

In practice, we compute the conductor exponents of $A/M$, where $A = E$ or $Z$. However, if $A/M$ has potentially good reduction, then the $\rhobar_{A,p} |_{I_M}$ factors through a finite group of order only divisible by $2$ and $3$. Hence, the conductor exponent of $\rhobar_{A,p} |_{G_M}$ is the same as the conductor exponent of $\rho_{A,p} |_{G_M}$, where $\rho_{A,p}$ denotes the $p$-adic representation attached to~$A$. This in turn coincides with the conductor exponent of $A/M$, provided $p \not= 2,3$.

Note that version 2 is obtained by taking $M = L_E$. Indeed, we get $G_M = I_L$ and  $\rhobar_{E,p} |_{G_M}$ will have conductor exponent 0 (because $E/M$ has good reduction)  whereas $\rhobar_{Z,p} |_{G_M}$ has non-zero conductor exponent (because $Z/M$ does not 
have good reduction).

\begin{remark}
In applications, the curves are often defined over a totally 
real number field $K$. Therefore, we can test if any of the 
versions above succeeds for different primes. Success at one 
prime is enough to discard the global isomorphism of two mod $p$ representations.
\end{remark}

\section{A multi-Frey approach to the equation $x^{5} + y^{5} = 3 z^p$}

In this section, we will use the following
factorization and notation
\[
x^5 + y^5 = (x+y)\phi_5(x,y) = (x+y)\psi_5(x,y)\bar{\psi}_5(x,y),
\]
where $\omega$ and $\bar{\omega}$ are the complex roots of $X^2+X-1$, and
\[
\phi_5(x,y) = x^4-x^3y+x^2y^2-xy^3+y^4
\]
\[
 \psi_5(x,y) = x^2 + \omega xy + y^2,  \qquad \bar{\psi}_5(x,y) = x^2 + \bar{\omega} xy + y^2.
\]

\subsection{The modular method over $\Q$}
\label{S:overQ}

Here we compile results from \cite{Bill} and \cite{BD}.

Let $a,b$ be coprime integers with~$a+b\not=0$. We consider the following Frey elliptic curve over~$\Q$ denoted~$E(a,b)$ or~$E$ in~\cite{Bill} and~\cite{BD}, and whose construction is due to Darmon~:
\[
 W_{a,b} \; : \; y^2 = x^3 - 5(a^2 + b^2)x^2 + 5\phi_5(a,b)x.
\]
The discriminant~$\Delta(W_{a,b})$ of~$W_{a,b}$ is given by
\[
\Delta(W_{a,b})= 2^4 5^3 (a+b)^2 (a^5 + b^5)^2.
\]
The following lemma is a reformulation of results proved in Section~$2$ of~\cite{Bill}.
\begin{lemma}\label{lemma:W}
The conductor~$N_{W_{a,b}}$ of~$W_{a,b}$ is 
\[
N_{W_{a,b}}=2^\alpha\cdot 5^2\cdot r,
\]
where $r$ is the product of all prime divisors~$\not=2,5$ of~$a^5+b^5$ and 
\[
\alpha=\left\{
\begin{array}{ll}
3 & \text{if $ab\equiv 0\pmod{4}$} \\
4 & \text{if $ab\equiv 2\pmod{4}$ or $\vv_2(a+b)=1$} \\
0 & \text{if $\vv_2(a+b)=2$} \\
1 & \text{if $\vv_2(a+b)\ge 3$} \\
\end{array}
\right..
\]
Furthermore, the following properties hold where $j(W_{a,b})$ denotes the $j$-invariant of~$W_{a,b}$~:
\begin{itemize}
\item  if $\ell\not=2,5$ is a prime of bad reduction, then the model defining~$W_{a,b}$ is minimal at~$\ell$ and we have $\vv_{\ell}(\Delta(W_{a,b}))=\delta\vv_\ell(a^5+b^5)$ where $\delta=2$, $4$ if $\ell$ divides~$\phi_5(a,b)$ or $\ell$ divides~$a+b$ respectively;
\item we have $\vv_2(j(W_{a,b}))\ge 0$ if and only if~$\vv_2(a+b)\le 2$;
\item we have $\vv_5(j(W_{a,b}))=1-4\vv_5(a+b)<0$ if~$5\mid a+b$ and $\vv_5(j(W_{a,b}))=0$ otherwise.
\end{itemize}
\end{lemma}

Let $W_0$ and~$W_0'$ be the rational elliptic curves defined by the following equations
\[
W_0 \; : \; y^2=x^3+x^2+592x-16812\quad\text{and}\quad W_0' \; : \;   y^2=x^3-x^2-333x-2088.
\]
They are labelled \cite[\href{http://www.lmfdb.org/EllipticCurve/Q/1200/k/8}{1200.k8}]{lmfdb} and~\cite[\href{http://www.lmfdb.org/EllipticCurve/Q/1200/a/1}{1200.a1}]{lmfdb} in LMFDB respectively. In~\cite{Bill}, the elliptic curves~$W_0$ and $W_0'$ were referred to as~$1200$P$1$ and $1200$N$1$ (in Cremona's labelling) respectively, whereas in~\cite{BD} the authors used Stein's notation~$1200$K$1$ and $1200$A$1$. 

\begin{proposition}\label{L:BD} 
Let $(a,b,c)$ be a non-trivial primitive
solution to \eqref{E:55p} for $p \geq 5$. Write $W=W_{a,b}$. Then we have~$p > 10^7$,  $\vv_2(ab) = 1$ and $5\nmid a+b$. Furthermore, we have $\rhobar_{W,p} \cong \rhobar_{W_0',p}$. 
\end{proposition}
\begin{proof}
According to~\cite[Remark~4.6]{BD}, we have~$p>10^7$. Besides, it follows from conductor computations (recalled above and in Section~$3$ of~\cite{Bill}) and \cite[Lemma~4.4]{BD}, that we have $\vv_2(ab)=1$ and $\rhobar_{W,p} \cong \rhobar_{W_0,p}$ or  $\rhobar_{W,p} \cong \rhobar_{W_0',p}$ according to whether $5$ divides~$a+b$ or not. 

The curve~$W_0$ has bad additive reduction at~$2$ with potentially multiplicative reduction. On the other hand, from $\vv_2(ab)=1$ and Lemma~\ref{lemma:W}, it follows that~$\vv_2(j(W))\ge 0$. Therefore if~$I_2$ denotes an inertia subgroup at~$2$, then $\#\rhobar_{W,p}$ belongs to~$\{2,3,4,6,8,24\}$ while by the theory of Tate curves we have~$\#\rhobar_{W_0,p}(I_2)=2$ or~$2p$. In particular, it follows from Version~$1$ of the `image of inertia argument' explained in Section~\ref{S:inertia} that $\rhobar_{W,p} \not\cong \rhobar_{W_0,p}$ and $5\nmid a+b$ as claimed.

Alternatively, we can argue as follows: Suppose $5 \mid a + b$ and $\rhobar_{W,p} \cong \rhobar_{W_0,p}$. Then it follows that $\rhobar_{W,p} \otimes \chi \cong \rhobar_{W \otimes \chi,p} \cong \rhobar_{W_0 \otimes \chi,p} \cong \rhobar_{W_0,p} \otimes \chi$, where $W \otimes \chi$ and $W_0 \otimes \chi$ are the twists of $W$ and $W_0$, respectively, by the quadratic character $\chi = \chi_{-1}$ associated to the quadratic field $\Q(\sqrt{-1})$. Now, the trace of Frobenius at $3$ of $W_0 \otimes \chi$ is $a_3(W_0 \otimes \chi) = -1$, where as the possible traces of Frobenius at $3$ of $W \otimes \chi$ are $a_3(W \otimes \chi) = 1, \pm 2$. Hence, $p \le 3$, a contradiction.

\end{proof}

\subsection{The modular method over $\Q(\sqrt{5})$}
\label{S:overQ5}

In~\cite{DF1}, the modular method was applied with the multi-Frey technique using two
Frey \mbox{$\Q$-curves} defined over $\Q(\sqrt{5})$ to solve~\eqref{E:55p} for a set
of prime exponents with Dirichlet density $3/4$. 
At the time, the purpose of using $\Q$-curves
was to guarantee their modularity. It is now known that elliptic
curves over real quadratic fields are modular (see \cite{FLHS}) and therefore we
can work directly over $\Q(\sqrt{5})$, largely simplifying the arguments.

We now sharpen the relevant results from \cite{DF1} in the language of Hilbert modular forms.

Let $a,b$ be coprime integers. Using the notation in the beginning of this section, we consider the two elliptic curves defined over~$\Q(\sqrt{5})$ by the following equations~:
\begin{eqnarray*}
E_{a,b} & : & y^2 = x^3 + 2(a+b)x^2 - \bar{\omega}\psi_5(a,b)x \\
F_{a,b} & : & y^2 = x^3 + 2(a-b)x^2 + \left(\frac{-3(\omega-\bar{\omega})}{10} + \frac{1}{2}\right)\psi_5(a,b)x.
\end{eqnarray*}
These two curves were denoted~$E_{(a,b)}$ and~$F_{(a,b)}$ in~\cite{DF1}, respectively. 
Their standard  invariants  are given by the following identities~:
\begin{eqnarray}
c_4(E_{a,b}) & = & -2^4\left(\bar{\omega}\psi_5(a,b)+2^2\omega\bar{\psi}_5(a,b)\right)  \label{eq:c4_formulaI}\\
c_6(E_{a,b}) & = & -2^6(a+b)\left(\bar{\omega}\psi_5(a,b)-2^3\omega\bar{\psi}_5(a,b)\right) \\
\Delta(E_{a,b}) & = & 2^6\bar{\omega}\phi_5(a,b)\psi_5(a,b) \label{eq:disc_formulaI} 
\end{eqnarray}
and
\begin{eqnarray}
c_4(F_{a,b}) & = & 2^4\left(\left(\frac{-3}{10}(\omega-\bar{\omega}) + \frac{1}{2}\right)\psi_5(a,b)+2^2\left(\frac{3}{10}(\omega-\bar{\omega}) + \frac{1}{2}\right)\bar{\psi_5}(a,b)\right) \label{eq:c4_formulaII}\\
c_6(F_{a,b}) & = & 2^6(a-b)\left(\left(\frac{-3}{10}(\omega-\bar{\omega}) + \frac{1}{2}\right)\psi_5(a,b)-2^3\left(\frac{3}{10}(\omega-\bar{\omega}) + \frac{1}{2}\right)\bar{\psi_5}(a,b)\right) \label{eq:c6_formulaII} \\
\Delta(F_{a,b}) & = & 2^6\left(\frac{-3}{10}(\omega-\bar{\omega}) + \frac{1}{2}\right)^2 \left(\frac{3}{10}(\omega-\bar{\omega}) + \frac{1}{2}\right)\phi_5(a,b)\psi_5(a,b). \label{eq:disc_formulaII}
\end{eqnarray}

We now determine the conductors of~$E_{a,b}$ and~$F_{a,b}$. For simplicity, let us write $E=E_{a,b}$ and $F=F_{a,b}$ and denote by $N_E$ and $N_F$ the conductors of the curves $E$ and~$F$, respectively.

\begin{lemma}	\label{lem:bad_reduction_55p_E_and_FI}
Let $C$ be one of the curves~$E$ or~$F$ and let~$\fq$ be a prime ideal in~$\Q(\sqrt{5})$ of residual characteristic~$\ell\neq2, 5$. Then $C$ has bad reduction at~$\fq$ if and only if $\ell\mid\phi_5(a,b)$. Moreover in that case, we have $\ell\equiv1\pmod{5}$ and 
$$
\vv_{\fq}(c_4(C))=0\quad\text{and}\quad \vv_{\fq}(\Delta(C))=2\vv_{\fq}(\psi_5(a,b))+\vv_{\fq}(\bar{\psi_5}(a,b))
$$
In particular, $C$ has bad multiplicative reduction at~$\fq$ and hence~$\vv_{\fq}(	N_C)=1$.
\end{lemma}
\begin{proof}
Recall that~$\phi_5(a,b)=\psi_5(a,b)\bar{\psi}_5(a,b)$ with~$\psi_5(a,b)$, $\bar{\psi_5}(a,b)$ coprime outside~$5$ (\cite[Prop.~3.1]{DF1}). 
If~$C$ has bad reduction at~$\fq$ with $\fq$ above $\ell\not=2,5$, then by formulas~\eqref{eq:c4_formulaI}-\eqref{eq:disc_formulaI}~and~\eqref{eq:c4_formulaII}-\eqref{eq:disc_formulaII}, we have that $\fq$ divides $\psi_5(a,b)\phi_5(a,b)=\psi_5(a,b)^2\bar{\psi_5}(a,b)$ and~$\fq\nmid c_4(C)$. Hence $\ell\mid\phi_5(a,b)$. Conversely, if $\ell\mid \phi_5(a,b)$, then any prime ideal~$\fq$ above~$\ell$ divides~$\psi_5(a,b)\bar{\psi}_5(a,b)$. In particular, we have $\fq\mid\Delta(C)$ and $\fq\nmid c_4(C)$.

Hence the result with the congruence~$\ell\equiv1\pmod{5}$ coming from~\cite[Lemma~2.2]{DF1}.
\end{proof}

Let $\fq_2$ and $\fq_5$ be the unique primes in $\Q(\sqrt{5})$ above $2$ and~$5$, respectively. Since $2$ is inert in $\Q(\sqrt{5})$ we will write simply $2$ for $\fq_2$. 
\begin{lemma}\label{lem:bad_reduction_55p_E_and_FII}
We have the following valuations~:
\begin{align}
  & \vv_2(N_E) = \vv_2(N_F) = 6, \label{cond5-1} \\
  & \vv_{\fq_5}(N_E) = 0 \text{ when } 5 \nmid a+b \text{ and } \vv_{\fq_5}(N_E) = 2 \text{ when } 5 \mid a+b, \\
  & \vv_{\fq_5}(N_F) = 2 \text{ when } 5 \nmid a+b \text{ and } \vv_{\fq_5}(N_F) = 0 \text{ when } 5 \mid a+b. \label{cond5-2}
\end{align}

\end{lemma}
\begin{proof}
We only give the details of our computations for the curve~$F$, the case of~$E$ being similar, but simpler. The given model for~$F$ is integral at~$2$ and we have $(\vv_2(c_4(F)),\vv_2(\Delta(F))) = (4,6)$ (see formulas~\eqref{eq:c4_formulaII} and~\eqref{eq:disc_formulaII}). Therefore, according to~\cite[Tableau~IV]{papado}, we are either in Case~$3$ or in Case~$4$ of Tate's classification. To decide which case actually occurs, we then apply Proposition~$1$ of \emph{loc. cit.} with, in its notation, $t=0$ and $r=1+\bar{\omega}$ or $r=1$ according to whether $ab$ is even or odd, respectively. Let us denote by $a_2$ and $a_4$ the coefficients of~$x^2$ and~$x$ in the right-hand side of the equation defining~$F$, respectively. Then, we have $\vv_2(a_4+ra_2+r^2)=1$ and we conclude that we are in Case~$3$ of Tate's classification. In particular, we have $\vv_2(N_E) = 6$.

For the conductor valuation at~$\fq_5$, we first notice that the given model for~$F$ is integral at~$\fq_5$ if and only if~$5$ divides~$a+b$. In that case we have $\vv_{\fq_5}(\phi_5(a,b))=2$ and $\vv_{\fq_5}(\psi_5(a,b))=1$. In  particular, the curve~$F$ has good reduction at~$\fq_5$ and therefore $\vv_{\fq_5}(N_F) = 0$. If $5$ does not divide~$a+b$, then we have $(\vv_{\fq_5}(c_4(F)),\vv_{\fq_5}(\Delta(F))) = (-1,-3)$. A change of variables over~$\Q(\sqrt{5})$	then gives an integral model for~$F$ whose~$c_4$ and~$\Delta$ invariants have respective valuations~$3$ and~$9$ at~$\fq_5$. According to~\cite[Tableau~I]{papado}, we have ${\vv_{\fq_5}(N_F)=2}$.
\end{proof}

In \cite{DF1}, the work of Ellenberg on $\Q$-curves (see \cite[Proposition~3.2]{Ell}) was 
used to establish that the mod~$p$ Galois representations attached to~$E_{a,b}$ and $F_{a,b}$ are
irreducible for $p=11$ and $p \geq 17$. We establish here an irreducibility result
without using the fact that $E_{a,b}$ and $F_{a,b}$ are $\Q$-curves.

\begin{proposition}\label{irreducible5}
 Let $p\ge 7$  be a prime number. Then,  $\rhobar_{E,p}$ and $\rhobar_{F,p}$ are irreducible when $5\nmid a+b$ and $5\mid a+b$ respectively.
\end{proposition}
\begin{proof}
Let $p \ge 7$ be a prime, and put $C = E$ or $C = F$.

Let us denote by $\rhobar_{C,p}^{ss}$ the semi-simplification of the representation~$\rhobar_{C,p}$. Suppose 
$\rhobar_{C,p}^{ss} \cong \theta \oplus \theta'$ with the characters~$\theta, \theta'$ satisfying 
$\theta \theta' = \chi_p$ where~$\chi_p$ denotes the mod~$p$ cyclotomic character. By \cite[Lemma~6.3]{FreSikFLTsmall} for instance, we have that $\theta$ and $\theta'$ are unramified outside $p$ and the additive primes of~$C$. Furthermore, $\theta$ and $\theta'$ have the same conductor away from $p$. The unit group of $K$ is 
generated by $\{-1, \epsilon\}$ where $\epsilon^2-\epsilon-1=0$. In the notation of~\cite{FS}, we compute $B = -2^6\cdot 5$. From the first paragraph of the proof of \cite[Theorem~1]{FS} we 
thus conclude that exactly one of $\theta$, $\theta'$ 
ramifies at (the primes above) $p$. Let us therefore assume that $\theta$ is unramified at~$p$.

Under the assumptions of the proposition, the only additive prime for~$C$ is~$2$ and it satisfies $\Norm(2) = 4$. It follows from $\vv_2(N_C) = 6$
and \cite[Theorem~1.5]{Jarv2} that level lowering at $2$ cannot occur.
Therefore, from the conductor computations above (see Lemmas~\ref{lem:bad_reduction_55p_E_and_FI} and~\ref{lem:bad_reduction_55p_E_and_FII}) it follows that the conductor of $\theta$ is $2^3$.

The Ray class group of~$\Q(\sqrt{5})$ of modulus~$2^3 \infty_1 \infty_2$ (where 
$\infty_1$ and $\infty_2$ denote the two real places) is isomorphic to~$(\Z/2\Z)^3$. In particular, if $I_2 \subset G_{\Q(\sqrt{5})}$ denotes an inertia subgroup at~$2$, then $\theta |_{I_2}$ is of order~$1$ or~$2$. Thus either $C$ or a quadratic twist 
$C'$ of $C$ has a~$2$-torsion point defined over $\Q(\sqrt{5})$. Therefore the torsion subgroup
of $C$ or $C'$ has order divisible by $2p$ with $p \geq 7$. 
From \cite[Theorem 6]{BruNaj}, we see that this is impossible.
\end{proof}

Since~$E_{a,b}$ and~$F_{a,b}$ are defined over a real quadratic field, they are modular by the main result of~\cite{FLHS}. This completes Step~$3$ of the modular method.

\subsection{Bounding the exponent}
\label{bounding}
We are now in position to study equation~\eqref{E:rrp} with $r=5$. Suppose that there exists an integer~$c$ such that $(a,b,c)$ is a non-trivial primitive solution to equation~\eqref{E:rrp} with $r=5$ and $p\geq 7$ and assume that all the prime factors $\ell$ of $d$
satisfy $\ell \not\equiv 1 \pmod{5}$. Write $E = E_{a,b}$ and $F = F_{a,b}$.

The following lemma summarizes Step~$4$ of the modular method as applied to~$E$ and~$F$.
\begin{lemma} 
There exist a Hilbert newform~$f$ over $\Q(\sqrt{5})$ of parallel weight $2$, trivial character
and level~$2^6$ and a prime ideal~$\fp$ above~$p$ in the coefficient field of~$f$ such that
\[
\rhobar_{f,\fp}\cong \rhobar_{F,p}, \qquad \text{or } \qquad \rhobar_{f,\fp}\cong \rhobar_{E,p},
\]
according to whether~$5$ divides~$a+b$ or not.
\label{L:DF}
\end{lemma}
\begin{proof}
Let $\fq$ be a prime ideal in~$\Q(\sqrt{5})$ of bad reduction for~$E$ or~$F$ with residual characteristic~$\ell\not=2,5$. According to Lemma~\ref{lem:bad_reduction_55p_E_and_FI} and~\cite[Lemma~2.2]{DF1}, the reduction is multiplicative and, by our assumption on~$d$, the valuation of the minimal discriminant at~$\fq$ is $\delta\vv_{\ell}(a^5+b^5)=\delta(\vv_{\ell}(d)+p\vv_{\ell}(c))=\delta p\vv_{\ell}(c)$, where~$\delta=1$, $2$ when $\fq$ divides~$\bar{\psi_5}(a,b)$ or $\fq$ divides~$\psi_5(a,b)$ respectively. In particular, it is divisible by~$p$. We conclude from Lemma~\ref{lem:bad_reduction_55p_E_and_FII} that the Artin conductor of the mod~$p$ representations of~$F$ and~$E$ is~$2^6$ according to whether~$5$ divides~$a+b$ or not. 

The rest of the proof follows by applying level lowering for Hilbert modular forms (see \cite{Fuj}, \cite{Jarv}, \cite{Raj}), with irreducibility coming from Proposition~\ref{irreducible5} above.
\end{proof}

Let $q \neq 2,5$ be a rational prime such that $q \not\equiv 1 \pmod{5}$ and let $\fq$ be a prime in $\Q(\sqrt{5})$ above~$q$. It follows from Lemma~\ref{lem:bad_reduction_55p_E_and_FI} that for any integers~$x,y$ with $(x,y)\not=(0,0)$ and~$0\leq x,y\leq q-1$, both elliptic curves~$E_{x,y}$ and~$F_{x,y}$ have good reduction at~$\fq$.  Moreover if $x,y$ are defined by $(a,b)\equiv(x,y)\pmod{q}$ and~$0\le x,y\le q-1$, then $a_{\fq}(E_{a,b})=a_{\fq}(E_{x,y})$. 

The result below follows from Lemma~\ref{L:DF} and our definitions here and in Section~\ref{S:Overview}.
\begin{proposition}
There exists a Hilbert newform~$f$ over $\Q(\sqrt{5})$ of parallel weight $2$, trivial character
and level~$2^6$ such that for any prime $q \neq 2,5$ with $q \not\equiv 1 \pmod{5}$, there exists $(x,y)\in\{0,\dots,q-1\}^2\setminus\{(0,0)\}$ such that we have $p\mid qB_q(E_{x,y},f)$ or $p\mid qB_q(F_{x,y},f)$ respectively if $5 \nmid a+b$ or $5 \mid a+b$.
\label{P:summary}
\end{proposition}

The following summarizes part of Step 5 of the modular method as applied to the Frey elliptic curves $E$ and $F$.
\begin{proposition}
\begin{itemize}
\item[]
 \item[(1)] If $5 \nmid a + b$ and $p\geq 7$ then $\rhobar_{E,p}$ is isomorphic to the mod $p$ representation of one of the curves
 \[E_{1,0}, \quad E_{1,0} \otimes \chi_{-1}, \quad E_{1,0} \otimes \chi_{2}, \quad
  \quad E_{1,0} \otimes \chi_{-2}, \quad E_{1,1}, \quad E_{1,1} \otimes \chi_2;
 \]
 \item[(2)] If $5 \mid a + b$ and $p\geq 11$ then $\rhobar_{F,p}$ is isomorphic to the mod $p$ representation
of one of the curves
\[ F_{1,-1} \qquad \text{or} \qquad F_{1,-1} \otimes \chi_2, \]
\end{itemize}
where $\chi_{D}$ denotes the quadratic character corresponding
to the field $\Q(\sqrt{D})$.
\label{P:generald}
\end{proposition}
\begin{proof}
Using \cite{programs}, we do the following: we compute all the
newforms over $\Q(\sqrt{5})$ of level~$2^6$, parallel weight 2 and trivial
character. For each such newform $h$, we compute $q\mathcal{E}_q(h)$ and $q\mathcal{F}_q(h)$
for all primes $q \leq 30$ as above where $\mathcal{E}_q(h)$ (resp.  $\mathcal{F}_q(h)$) is the product of all~$B_q(E_{x,y},h)$ (resp.~$B_q(F_{x,y},h)$) over the pairs~$(x,y)\not=(0,0)$ of integers  in the range~$\{0,\dots, q-1\}$.

Suppose $5 \nmid a+b$. From the previous proposition it follows that, for each $h$, 
if $p$ does not divide the gcd of all $q\mathcal{E}_q(h)$ we can discard $h$ for that $p$. 
This allows us to discard all except $6$ newforms for $p \geq 7$ ;  we identify the remaining $6$ newforms with twists of the Frey elliptic curves~$E_{1,0}$ and $E_{1,1}$.

Suppose $5 \mid a+b$. From the previous proposition it follows that, for each $h$, 
if $p$ does not divide the gcd of all $q\mathcal{F}_q(h)$ we can discard $h$ for that $p$. 
For $p \geq 11$, this allows to discard all except $2$ newforms which correspond to $F_{1,-1}$ and its quadratic twist by~$2$. We also note for later use that~$p\ge 7$ works for all except three other newforms~$f$, all of them satisfying~$a_3(f)=4$.
\end{proof}

\subsection{Proof of Theorem~\ref{T:main5}}
\label{S:proofMain5}
We will now prove Theorem~\ref{T:main5} under the slightly more general situation where $d$ is divisible by~$3$ but not by any prime $\ell \equiv 1 \pmod 5$.

From $3 \mid d$ and \cite[Lemma~2.2]{DF1}, it follows that $3 \mid a+b$. This imposes a very strong restriction on the value of the trace of Frobenius of $E_{a,b}$ at (the unique prime ideal above)~$3$ in~$\Q(\sqrt{5})$. Namely, the elliptic curve $E_{a,b}$ reduces modulo~$3$ to the curve defined by $y^2=x^3-\bar{\omega}^2x$. Hence, we have $a_3(E_{a,b})=6$.

Note that the elliptic curves~$E_{x,y}$ that appear in part~(1) of Proposition~\ref{P:generald}  satisfy $x+y\not\equiv0\pmod{3}$. Therefore, for $p\geq7$, one may hope to discard them by computing their trace of Frobenius at~$3$. Indeed, we find that the $a_3$ coefficient of the curves $E_{1,0}$, $E_{1,0} \otimes \chi_{-1}$, $E_{1,0} \otimes \chi_{2}$, $E_{1,0} \otimes \chi_{-2}$, $E_{1,1}$, and $E_{1,1} \otimes \chi_2$ is~$4$. We have thus proved the following result.
\begin{proposition}
\label{contra2}
If~$p\ge 7$ and $d$ is divisible by~$3$ but not by any prime $\ell \equiv 1 \pmod 5$, then we necessarily have $5\mid a + b$.
\end{proposition}

\begin{remark}
The previous type of argument does not always work; for instance, when $r = 7$ and $d = 3$ in equation~\eqref{E:rrp}, the condition $3 \mid a + b$ does not distinguish $E_{0,1}$ and $E_{1,-1}$ by traces of Frobenius at $3$, where $E_{a,b}$ is the Frey elliptic curve in the last paragraph of \cite[p.\ 630]{F}.
\end{remark}

To prove Theorem~\ref{T:main5} it now suffices to notice that the cases $p = 2$, $p=3$ and~$p=5$ follow from \cite[Theorem 1.1]{BenSki}, \cite[Theorem 1.5]{BenVatYaz} and~\cite[Th\'eor\`eme IX]{Di28} respectively. Hence we can assume~$p\geq7$. Applying Propositions~\ref{L:BD}~and~\ref{contra2} concludes the proof.

\begin{remark}
Note we cannot improve on the result in \cite{DF1} for $r = 5$ and $d = 2$ since we do not have the condition $3 \mid a + b$ to eliminate the curves in Proposition~\ref{P:generald} (1); furthermore, the additional use of the Frey elliptic curve $W_{a,b}$ also does not help because $W_{1,1}$ is an elliptic curve without complex multiplication.
\end{remark}

\section{Partial results for $x^5 + y^5 = d z^p$ with $d = 1, 2$}
\label{partial}

It is sometimes possible to resolve equation $\eqref{E:rrp}$ by assuming additional $q$-adic conditions to avoid the obstructing trivial solutions. 
In this section we provide such examples regarding the equation  
\begin{equation}
x^5 + y^5 = dz^p, \qquad \text{ where } \quad d \in \{ 1, 2 \}.
\label{E:55pd}
\end{equation}
First note that the conditions on $c$ of Theorem~\ref{T:partial} can easily be translated 
into divisibility conditions on $a+b$. More precisely, Theorem~\ref{T:partial}
follows from the following two theorems.
\begin{theorem}
\label{T:main5a} Assume $d = 1,2$. Then, for all primes $p$, 
there are no non-trivial primitive solutions $(a,b,c)$
to~\eqref{E:55pd} satisfying $5 \mid a + b$.
\end{theorem}
\begin{theorem}\label{T:main5b}
Assume $d=1$ (resp.\ $d=2$). Then, for all primes $p$, 
there are no non-trivial primitive solutions to~\eqref{E:55pd} satisfying
$2 \mid a + b$ (resp.\ $4 \mid a + b$).
\end{theorem}

We want to emphasize that, in the proof of
Theorem~\ref{T:main5a}, using the multi-Frey technique we are able to force a Frey
curve to have multiplicative reduction at 3. 

These results, and their proofs, should illustrate clearly to the reader
that the obstruction to solving~\eqref{E:55pd} with $d = 1$ (resp.\ $d = 2$) is that none of the Frey curves we use are sensitive to the trivial solutions $\pm (1,0,1), \pm (0,1,1)$ (resp.\  $\pm (1,1,1)$).

\subsection{Proof of Theorem~\ref{T:main5a}}
\label{S:byproduct2}

The cases $p = 2$ and $p=3$ follow from \cite[Theorem 1.1]{BenSki}, \cite[Theorem 1.5]{BenVatYaz},  respectively. It follows from Fermat's Last Theorem and the main theorem of \cite{DarmonMerel} 
that the result holds for $p=5$. Hence we can assume~$p\geq7$.

Let $(a,b,c)$ be a putative non-trivial primitive solution to equation
\eqref{E:55pd} with $d = 1, 2$, exponent $p \geq 7$ and $5 \mid a + b$.

By part (2) of Proposition~\ref{P:generald} we have 
$\rhobar_{F_{a,b},p} \cong \rhobar_{A,p}$, 
where $A = F_{1,-1}$ or $F_{1,-1} \otimes \chi_2$ when $p \geq 11$; furthermore, from its proof 
it follows that for $p=7$ we can have $\rhobar_{F,p} \cong \rhobar_{A,p}$ or 
$\rhobar_{F,p} \cong \rhobar_{f,p}$, where $f$ is one of other four possible Hilbert newforms over~$\Q(\sqrt{5})$ of parallel weight~$2$, trivial character
and level~$2^6$.

The traces of Frobenius at 3 of these six newforms satisfy $a_3(A) = a_3(f) = 4$;
using \texttt{Magma} to compute $a_3(F_{a,b})$ shows that $3
\mid a + b$ (if not, then $a_3(F_{a,b}) \in\{-2,6\}$ and we get that $p \mid 6$, which is not the case). This means the curve $W = W_{a,b}$ from Section~\ref{S:overQ}
has multiplicative reduction at $3$ (see Lemma~\ref{lemma:W}). Note that this is another instance of using the multi-Frey technique.

From \cite[Proposition~3.1]{Bill} we have that the representation~$\rhobar_{W,p}$ is irreducible.
A standard application of the modular method with $W$ (which follows from Propositions~3.3 and~3.4 of ~\cite{Bill}) gives
that $\rhobar_{W,p} \cong
\rhobar_{g,p}$, where $g$ is a rational newform of weight~$2$, trivial Nebentypus and level $2^4 \cdot 5^2$, $2^3 \cdot 5^2$, or $2 \cdot 5^2$ for $d = 1$ and $2^4
\cdot 5^2$, $2 \cdot 5^2$ for $d = 2$ respectively. All newforms in these spaces correspond to (isogeny classes of) elliptic curves over $\Q$. Since level lowering is happening at the prime 3, 
we must have that $p \mid (3+1)^2 - a_3(g)^2$. By the Hasse bound and our assumption, it implies  $p=7$  and $a_3(g)=\pm3$.

We then notice using~\cite{Cre97} that there are four newforms~$g$ of these levels for which we have~$a_3(g)=\pm3$. Moreover they all correspond to elliptic curves with potentially good reduction at~$5$ and whose minimal discriminant has valuation~$2$ or~$8$ at~$5$. According to~\cite[p.~312]{Ser72} it follows that $\#\rhobar_{g,7}(I_5) = 3$  or $6$, where $I_5$ is an inertia subgroup at 5. 

On the other hand, since $5 \mid a+b$, the curve~$W$ has potentially multiplicative reduction at~$5$
(see Lemma~\ref{lemma:W}). Hence by the theory of Tate curves, we have $\#\rhobar_{W,7}(I_5) = 2$ or $14$.
According to Version~$1$ of the `image of inertia argument' explained in Section~\ref{S:inertia}, this gives the desired contradiction.

\subsection{Proof of Theorem~\ref{T:main5b}}
\label{S:byproduct1}
As in the previous proof, the result is known for $p \le 5$.
Let $(a,b,c)$ be a putative non-trivial primitive solution to equation
\eqref{E:55pd} with $d = 1$ and $2 \mid a + b$ (resp.\ $d = 2$ and $4 \mid a +
b$) for $p \geq 7$. 

In the case $d = 1$, the condition $2 \mid a + b$ implies that in fact $8
\mid a + b$, because $2 \mid c$, $p \geq 7$ and $2 \nmid \phi_5(a,b)$, where
we recall $a^5 + b^5 = (a+b)\phi_5(a,b) = dc^p$;
in the case $d = 2$, the condition $4 \mid a + b$ also implies that
in fact $8 \mid a + b$. So  we now assume $8 \mid a + b$.

By Theorem~\ref{T:main5a}, we may assume $5 \nmid a+b$, and then invoking
part (1) of Proposition~\ref{P:generald} we deduce $\rhobar_{E,p} \cong \rhobar_{A,p}$
where $A = E_{1,0}$,  $E_{1,0} \otimes \chi_{-1}$, $E_{1,0} \otimes
\chi_{2}$, $E_{1,0} \otimes \chi_{-2}$, $E_{1,1}$, or $E_{1,1} \otimes \chi_2$.

The result now follows from version 2 of the image of inertia argument (see Section~\ref{S:inertia}).
Indeed, from $\rhobar_{E,p} \cong \rhobar_{A,p}$ we know that the inertial 
field at $2$ of $E$ and $A$ must be the same. 
By Proposition~\ref{P:fieldsAt2r5} below
and the assumption $8 \mid a + b$, we see 
this is not possible, as desired.

Write $L_{a,b} = L_{E_{a,b}}$ for the inertial field at $2$ corresponding
to the Frey elliptic curve $E_{a,b}$ (i.e.\ the field fixed by the kernel of
$\rhobar_{E_{a,b},m}(I_2)$ for any $m \ge 3$ coprime to $2$).
Respectively, for any integers $x,y$, we write $L_{x,y,D}$ for the inertial field at~$2$
corresponding to the curve~$E_{x,y} \otimes \chi_D$.

\begin{proposition} Suppose $(a,b,c)$ is a non-trivial primitive solution to
\eqref{E:55pd} satisfying $8 \mid a + b$. Then $L_{a,b} \not=  L_{1,0},
\; L_{1,0,-1}, \; L_{1,0,2}, \; L_{1,0,-2}, \; 
L_{1,1}, \; L_{1,1,2}$. \label{P:fieldsAt2r5}
\end{proposition}
\begin{proof}
This is verified using \cite{programs} by considering a suitable subfield~$M$ of the $3$-division field of~$Z$ over $\Q(\sqrt{5})$, where $Z = E_{1,0}$,  $E_{1,0} \otimes \chi_{-1}$, $E_{1,0} \otimes \chi_{2}$, $E_{1,0} \otimes \chi_{-2}$, $E_{1,1}$, or $E_{1,1} \otimes \chi_2$, with the property that $Z$ has good reduction at a prime above~$2$ of~$M$, but $E_{a,b}$ does not have good reduction at this prime above $2$ of $M$ if $8 \mid a + b$.  It turns out that we can take~$M$ to be the subfield generated by the~$x$ and~$y$ coordinates of a choice of a $3$-torsion point of~$Z$.

We have the following two cases:
\begin{itemize}
\item[(a)] For $Z = E_{1,1}, E_{1,1} \otimes \chi_2$, the choice of $M$ has degree $4$ over $\Q(\sqrt{5})$. Let $\fq'$ be the unique prime above $2$ of $M$ with ramification index $4$.
\item[(b)] For $Z = E_{1,0}, E_{1,0} \otimes \chi_{-1},E_{1,0} \otimes \chi_2,E_{1,0} \otimes \chi_{-2}$, the choice of  $M$ has degree $8$ over $\Q(\sqrt{5})$. Let $\fq'$ be the unique prime above $2$ of $M$ with ramification index $8$.
\end{itemize}
We remark the the full $3$-division field of $Z$ has degree $8$ and $48$ over $\Q(\sqrt{5})$, in cases (a) and (b), respectively. Thus, the choice of the smaller subfield $M$ makes the computation feasible in case (b).


Let~$a'$, $b'$ be coprime integers and consider the base change to~$M$ of~$E = E_{a,b}$ and~$E'=E_{a',b'}$. We have~$v_{\fq'}(\Delta(E)) = v_{\fq'}(\Delta(E')) = 24$ and $48$, in cases (a) and (b), respectively. 

Write~$a_4$, $a_4'$, $a_6$, $a_6'$ for~$a_4(a,b)$, $a_4(a',b')$, $a_6(a,b)$, $a_6(a',b')$ respectively and recall that~$a_4$, $a_4'$ and~$a_6$, $a_6'$ are homogeneous polynomials of degree~$4$ and~$6$ with coefficients in the integer ring of~$\Q(\sqrt{13})$, respectively.

Set~$k = 2$ and~$k = 4$ in cases (a) and (b), respectively. Suppose that the reduction type of $E'$ is either $II$, $II^*$ or $I_0^*$ and we have that $v_{\fq'}(a_4 - a'_4) \geq 4\cdot k$ and $v_{\fq'}(a_6 - a_6') \geq 6\cdot k$. By \cite[Lemma 2.1]{BeChDaYa} applied to~$E$ and~$E'$ over~$M$ with~$k$, the reduction types of $E$ and $E'$ at $\fq'$ are the same, and hence the conductor exponents at $\fq'$ of $E$ and $E'$ are the same.


Therefore, if $(a,b) \equiv (a',b') \pmod{2^3}$ holds, then the assumptions of the previous paragraph are satisfied and then the conductor exponent at $\fq'$ of $E_{a,b}$ is the same as that of $E_{a',b'}$, provided the reduction type of $E_{a',b'}$ is $II$, $II^*$ or $I_0^*$. Assuming $8 \mid a+b$ and using \cite{programs}, it is thus shown that $E_{a,b}/M$ has conductor exponent $\not= 0$ at the prime of $M$ above $2$, whereas $Z/M$ has good reduction at the prime of $M$ above $2$.
\end{proof}

\section{A result on the second case}\label{S:second_case}

In this section, we prove Theorem~\ref{T:second case}. The following proposition is known to experts, but we have not been able to find a suitable reference for it, so we include a proof.

\begin{proposition}  Let $f$ be a (classical) newform of weight $2$, trivial character and level $N$.
 Let $p$ be an odd prime not dividing $N$, and let $a_p$ denote the $p$-th Fourier coefficient of $f$. Then, a necessary condition for
the existence of a congruence between the $p$-adic Galois representation attached
 to $f$ and the one attached to a newform $g$ of level $pN$, trivial character and weight~$2$ is~:
 \begin{equation*} \label{P:levelraising}
 a_p \equiv \pm 1 \pmod{p}. 
\end{equation*}
\end{proposition}
\begin{proof}
Denote by $\rho_{f,p}$ and $\rho_{g,p}$ the restrictions to $\Gal(\overline{\Q}_p/\Q_p)$ of the
respective global $p$-adic Galois representations attached to $f$ and $g$. Then $f$ congruent to $g$ modulo $p$
implies in particular that the semi-simplifications of the residual local representations  $\bar{\rho}_{f,p}^{ss}$
and $\bar{\rho}_{g,p}^{ss}$
of, respectively, $\rho_{f,p}$ and $\rho_{g,p}$ are isomorphic.
We assume $p>2$. Since $\rho_{g,p}$  is semistable non-crystalline of weight $2$,  $\bar{\rho}_{g,p}^{ss}$ is reducible and isomorphic to:
$\chi_p \unr(\mu) \oplus \unr(\mu)$ for some mod~$p$ unit~$\mu$ and where $\chi_p$ denotes the mod~$p$ cyclotomic character (this is the case $k=2$ of \cite[Th\'eor\`eme~1.2]{BM}). Thus, the same holds for $\bar{\rho}_{f,p}^{ss}$. By \cite[Th\'eor\`eme~6.7]{Br} (a theorem that puts together results of Deligne, Serre, Fontaine and Edixhoven) 
this forces $a_p$ to be congruent to~$\pm 1$ modulo~$p$.
\end{proof}

Using the above proposition, we now prove Theorem~\ref{T:second case}.

The cases $p = 2$ and $p=3$ follow from \cite[Theorem 1.1]{BenSki}, \cite[Theorem 1.5]{BenVatYaz},  respectively. It follows from Fermat's Last Theorem and the main theorem of \cite{DarmonMerel} 
that the result holds for $p=5$. Hence we can assume~$p\geq7$.

We know (see \cite[Proposition~3.1]{Bill}) that the mod $p$ Galois representation attached to the Frey elliptic curve $W$ is irreducible, for every $p \ge 7$.  By level lowering, we have a congruence modulo~$p$ between the Frey elliptic curve~$W$ and some weight~$2$ newform of level $N=50$, $200$ or $400$. Since we are assuming that $p$ divides~$c$, level raising at~$p$ mod~$p$ is happening for this specific newform. This  implies in particular that the necessary condition in Proposition~\ref{P:levelraising} must hold. 

All newforms in these spaces correspond to (isogeny classes of) elliptic curves over $\Q$, and we consider the cases when:
\begin{enumerate}
\item the elliptic curve does not have a rational $2$-torsion point, or 
\item the elliptic curve has a rational $2$-torsion point.
\end{enumerate}

Case (1):  For all such elliptic curves, it can be checked (using \cite{Cre97} for instance) that the coefficient $a_3$ equals $\pm 1$ or $\pm 3$. Then we easily conclude using the congruence between these values and $0, \pm 2, \pm 4$ that this can not happen for $p > 7$. We are using the fact that the Frey elliptic curve $W$ has a rational $2$-torsion point, and we are covering both the cases of $W$ having good or multiplicative reduction at $3$. For $p=7$, the congruence forces $a_3=\pm3$. We then quickly verify using~\cite{Cre97} that none of the curves of level~$N\in\{50,200,400\}$ satisfies both $a_3=\pm3$ and $a_7\equiv\pm1\pmod{7}$.

Case (2): The fact that mod $p$ we have level raising at $p$ forces the necessary condition in Proposition~\ref{P:levelraising} to hold: $a_p \equiv \pm 1 \pmod p$. For an elliptic curve, this is equivalent to implies $a_p = \pm 1$ by the Hasse bound. But all curves in case (2) have a rational $2$-torsion point, thus all their coefficients $a_q$ for $q \nmid N$ are even. This gives a contradiction.

\begin{remark}
As pointed out to us by the referee, instead of Proposition~\ref{P:levelraising}, we could have used~\cite[Proposition~3 (iii)]{KraOes} because $f$ and $g$ correspond to elliptic curves over $\Q$. However, the more general Proposition~\ref{P:levelraising} may be useful in other applications of the modular method when level-lowering results in a newform with non-rational fourier coefficients.
\end{remark}

\section{A multi-Frey approach to the equation $x^{13} + y^{13} = 3z^p$}

Let $\zeta_{13}$ be a primitive $13$-th root of unity. In this section, we will use the following
factorization and notation
\begin{equation}
x^{13} + y^{13} = (x+y)\phi_{13}(x,y) = (x+y) \psi_{13}(x,y) \bar{\psi}_{13}(x,y),
\end{equation}
where
\begin{eqnarray*}
   \psi_{13}(x,y)   & = &  (x+\zeta_{13} y)(x+\zeta_{13}^4y)(x+\zeta_{13}^3y)(x+\zeta_{13}^{12}y)(x+\zeta_{13}^9y)(x+\zeta_{13}^{10}y) \\
 & = &  x^6 + \frac{1}{2} (w - 1) x^5 y + 2 x^4 y^2 + \frac{1}{2} (w + 1) x^3 y^3 + 2 x^2 y^4 + \frac{1}{2} (w - 1) x y^5 + y^6,
\end{eqnarray*}
and $\bar {\psi}_{13}(x,y)$ are the two degree $6$ irreducible factors of $\phi_{13}(x,y)$ over $\Q(w)$, where $w \in \Q(\zeta_{13})$ satisfies $w^2 = 13$.

\subsection{The modular method over $\Q(\sqrt{13})$}
\label{S:overQ13}

We will now prove the following theorem
by sharpening the methods in \cite{DF2} plus 
a refined image of inertia argument.

\begin{theorem} 
Let~$(a,b,c)$ be a non-trivial primitive solution to equation~\eqref{E:rrp} with $r=13$, $p \ge 5$, $p \not= 7,13$  and $d$ such that all its prime factors~$\ell$ satisfy $\ell\not\equiv1\pmod{13}$. If $3$ divides~$d$, then we have
\begin{itemize}
 \item[(A)]   $13 \mid a+b$ and
 \item[(B)] $4 \mid a+b$.
\end{itemize}
\label{T:4nmidaplusb}
\end{theorem}

Before entering the proof of this result, 
we first introduce tools from \cite{DF2} which are valid
beyond the setting  of the theorem. Write $\zeta=\zeta_{13}$ and define 
\[
 A_{x,y} = \alpha(x^2 + (\zeta +\zeta^{-1})xy+y^2), \quad
 B_{x,y} = \beta (x^2 + (\zeta^3 +\zeta^{-3})xy+y^2), \quad 
 C_{x,y} = \gamma (x^2+(\zeta^4+\zeta^{-4})xy+y^2), 
\]
where
\[
 \alpha = \zeta^4 + \zeta^{-4} - \zeta^{3} - \zeta^{-3}, \quad
 \beta  = \zeta + \zeta^{-1} - \zeta^{4} - \zeta^{-4}, \quad 
 \gamma = \zeta^3 + \zeta^{-3} - \zeta- \zeta^{-1}
\]
all have norm $13^2$. We note that $A_{x,y}$, $B_{x,y}$, $C_{x,y}$ are polynomials with coefficients in (the maximal totally real subfield of)~$\Q(\zeta)$ satisfying
$A_{x,y} + B_{x,y} + C_{x,y} = 0$.

Suppose now that $a,b$ are coprime integers. Let us denote by~$E_{a,b}$ the short Weierstrass model of the elliptic curve  $
 Y^2 = X(X-A_{a,b})(X+B_{a,b})$ given by
\[ 
E_{a,b} : y^2 = x^3 + a_4(a,b)x + a_6(a,b),
\]
where
\begin{eqnarray*}
 a_4(a,b) & =  & 3^3\left(A_{a,b}B_{a,b}+A_{a,b}C_{a,b}+B_{a,b}C_{a,b}\right)  \\
 a_6(a,b) & = &  -3^3\left(2A_{a,b}^3+3A_{a,b}^2B_{a,b}-3A_{a,b}B_{a,b}^2-2B_{a,b}^3\right).
\end{eqnarray*}
This curve (in a slightly different short model) was first considered in~\cite{DF2} where it is denoted~$E_0$. We then verify that $E_{a,b}$ is defined over~$\Q(\sqrt{13})$. Its standard invariants are  given by the following identities~:
\begin{eqnarray*}
c_4(E_{a,b}) & = & -2^4\cdot3\cdot a_4(a,b) = -2^4\cdot3^4 \left(A_{a,b}B_{a,b}+A_{a,b}C_{a,b}+B_{a,b}C_{a,b}\right)\\
\Delta(E_{a,b}) & = & 2^{4} \cdot 3^{12}  \left(A_{a,b}B_{a,b}C_{a,b}\right)^2 =  2^4\cdot 3^{12} \cdot 13 \cdot \psi_{13}(a,b)^2.
\end{eqnarray*}

We now determine the conductor of~$E_{a,b}$. For simplicity, let us write $E=E_{a,b}$ and~$N_E$ for its conductor.

\begin{lemma}\label{lem:bad_reduction_1313p_EI}
Let~$\fq$ be a prime ideal in~$\Q(\sqrt{13})$ of residual characteristic~$\ell\neq2, 13$. Then $E$ has bad reduction at~$\fq$ if and only if $\fq$ divides~$\psi_{13}(a,b)$. Moreover, in that case, we have
$$
\ell\equiv1\pmod{13},\quad \vv_{\fq}(c_4(E))=0\quad\text{and}\quad \vv_{\fq}(\Delta(E))=2\vv_\fq(\psi_{13}(a,b)).
$$
In particular, $E$ has bad multiplicative reduction at~$\fq$ and hence~$\vv_{\fq}(N_E)=1$.
\end{lemma}
\begin{proof}
Recall that as elements of~$\Q(\zeta)$, $A_{a,b}$, $B_{a,b}$ and~$C_{a,b}$ are relatively prime outside~$13$.

Let us first assume that~$\ell\not=3$. It follows from the formulas above that if $E$ has bad reduction at~$\fq$, then $\fq$ divides~$\psi_{13}(a,b)$ and $\fq$ does not divide~$c_4(E)$. Conversely if $\fq\mid\psi_{13}(a,b)$, then $\fq$ divides~$(A_{a,b}B_{a,b}C_{a,b})^2=13 \cdot \psi_{13}(a,b)^2$ and~$\fq$ does not divide~$c_4(E)$. Hence the equivalence. Moreover, we have $\vv_{\fq}(\Delta(E))=2\vv_\fq(\psi_{13}(a,b))$ and the congruence $\ell\equiv1\pmod{13}$ follows from \cite[Section 2]{DF2}.

It remains to show that~$E$ has good reduction at the prime ideals in~$\Q(\sqrt{13})$ above~$3$. Let~$\fq$ be such a prime. From \cite[Section 2]{DF2}, we have that~$\fq$ does not divide~$\psi_{13}(a,b)$. Therefore, we have  $(\vv_{\fq}(c_4(E)),\vv_{\fq}(\Delta(E)))=(\ge 4,12)$ and the defining model of~$E$ is not minimal at~$\fq$ (\cite[Tableau~I]{papado}). A change of variables then shows that~$E$ has good reduction at~$\fq$, as claimed.
\end{proof}

The following lemma follows from a similar statement for the curve~$E_0$ in~\cite[Prop.~3.3]{DF2}.
\begin{lemma}\label{lem:bad_reduction_55p_EII}
We have 
\[
\vv_w(N_E)=2\quad\text{and}\quad \vv_2(N_E)=s  \quad \text{ where } \quad s = 3, 4.
\]
Moreover, when $a+b$ is even, $s=3$ if  $4 \mid a+b$ and $s=4$ if $4 \nmid a+b$. 
\end{lemma}

The next proposition gives us the required irreducibility of the mod $p$
representation of~$E$.

\begin{proposition} \label{P:irredE13}
Let $p \geq 7$ be a prime number. Then the representation $\overline{\rho}_{E,p}$ is irreducible.

Moreover, if  $3 $ divides $ a+b$, then $\rhobar_{E,5}$ is also irreducible.
\end{proposition}
\begin{proof} 
We note that the proof of \cite[Theorem~3]{FS} applies in our situation, 
therefore proving the proposition for $p=11$ and $p \geq 17$. Assume therefore that $p\in\{5,7,13\}$.

We note that 3 splits in $\Q(\sqrt{13})$ and let $\fq_1$, $\fq_2$ be the primes above it with $w+1\in\fq_1$ (and $w-1\in\fq_2$).
By Lemma~\ref{lem:bad_reduction_1313p_EI}, the primes $\fq_1$ and $\fq_2$ are primes of good reduction of $E$;
since $a,b \in \Z$ we can check that the pairs of traces of Frobenius at these primes 
$(a_{\fq_1}(E),a_{\fq_2}(E))$ satisfy
\[
 (a_{\fq_1}(E),a_{\fq_2}(E)) \in \{(-3,-1),(-1,-3),(-1,1)\}.
\]
Moreover, the case $(-3,-1)$ occurs precisely when $3 \mid a+b$.
Therefore, we can compute the corresponding pairs of characteristic polynomials
of $(\rhobar_{E,p}(\Frob_{\fq_1}),\rhobar_{E,p}(\Frob_{\fq_2}))$ which are
given by 
\begin{equation}
\label{paircharacteristic}
  (x^2 - a_{\fq_1}(E) x + 3, x^2 - a_{\fq_2}(E) x + 3).
\end{equation} 

Now suppose that $\rhobar_{E,p}$ is reducible. Then, for any prime $\fq$ in $\Q(\sqrt{13})$ of good reduction of  $E$, the characteristic polynomial of $\rhobar_{E,p}(\Frob_\fq)$  must factor over $\F_p$ into two linear polynomials. In particular, this holds for $\fq = \fq_1, \fq_2$.
 
 For $p=5$, $7$, and $13$, we check that each of the pairs of polynomials in \eqref{paircharacteristic} always contains one polynomial that does not factor over $\F_p$ except when $p=5$ and $(a_{\fq_1}(E),a_{\fq_2}(E)) = (-1,1)$. This proves the proposition for $p \geq 7$. Finally, assume $3 \mid a+b$. In that case, we already observed that $(a_{\fq_1}(E),a_{\fq_2}(E)) = (-3,-1) \neq (-1,1)$. We conclude $\rhobar_{E,5}$ is irreducible, finishing the proof.
\end{proof}

We note that modularity of~$E_{a,b}$ is guaranteed by~ \cite{FLHS}, hence completing Step~$3$ of the modular method.

We are now in position to study equation~\eqref{E:rrp} with $r=13$. Suppose that there exists an integer~$c$ such that $(a,b,c)$ is a non-trivial primitive solution to \eqref{E:rrp} with $r=13$ and $p\geq 5$ and assume that all the prime factors $\ell$ of $d$ satisfy $\ell \not\equiv 1 \pmod{13}$. Write again~$E = E_{a,b}$.

The following lemma summarizes Step~$4$ of the modular method.
\begin{lemma} We have
\begin{equation*}
  \rhobar_{E,p} \cong \rhobar_{f,\fp}, 
\end{equation*}
where~$\fp$ is a prime in~$\Qbar$ of residual characteristic~$p$ and $f$ is a Hilbert newform over $\Q(\sqrt{13})$ of parallel weight $2$, trivial character 
and level 
\[
N_f = 2^{s}w^{2}, \quad \text{ where } \quad s = 3, 4.
\]
Moreover, when $a+b$ is even, $s=3$ if  $4 \mid a+b$ and $s=4$ if $4 \nmid a+b$. 

If in addition we have that $3$ divides~$a+b$, then the above also holds for $\rhobar_{E,5}$.
\label{L:Econductors}
\end{lemma}
\begin{proof}
Let $\fq$ be a prime ideal in~$\Q(\sqrt{13})$ of bad reduction for~$E$ with residual characteristic~$\ell\not=2,13$. According to Lemma~\ref{lem:bad_reduction_1313p_EI} and Section~$2$ of~\cite{DF2}, the reduction is multiplicative and, by our assumption on~$d$, the valuation of the minimal discriminant at~$\fq$ is $2\vv_\fq(\psi_{13}(a,b))=2\vv_{\ell}(a^{13}+b^{13})=2\vv_{\ell}(d)+2p\vv_{\ell}(c)=2p\vv_{\ell}(c)$. In particular, it is divisible by~$p$. We conclude from Lemma~\ref{lem:bad_reduction_55p_EII} that the Artin conductor of the mod~$p$ representations of~$E$ is~$2^sw^2$ where $s$ is valuation at~$2$ of the conductor of~$N_E$ (computed in Lemma~\ref{lem:bad_reduction_55p_EII}). 

The rest of the proof follows by applying level lowering for Hilbert modular forms (see \cite{Fuj}, \cite{Jarv}, \cite{Raj}), with irreducibility coming from Proposition~\ref{P:irredE13} above.
\end{proof}

The following summarizes part of Step 5 of the modular method as applied to the Frey elliptic curve~$E$.
\begin{proposition}\label{prop:irredE} 
Assume~$p \ge 7$ and $p \not= 13$.  Then, we have
\begin{equation} 
\rhobar_{E,p} \cong \rhobar_{Z,p},
\label{E:trivialisos}
\end{equation}
where $Z$ is one of the following elliptic curves
\[ E_{1,-1}, \quad E_{1,0} \quad \text{ or } \quad E_{1,1}.\]
In the case $p = 7$, we have an additional possibility that $\rhobar_{E,p} \cong \rhobar_{g,\fp_7}$ for a Hilbert newform $g$ over $\Q(\sqrt{13})$ of parallel weight $2$, trivial character, and level $2^3 w^2$, with field of coefficients $\Q(\sqrt{2})$, and a choice of prime $\fp_7$ above $7$ in this field.

If in addition $3$ divides~$a+b$, then we also have $\rhobar_{E,5} \cong \rhobar_{Z,5}$ for $Z$ as above.
\label{P:iso2}
\end{proposition}
\begin{proof}
Using \cite{programs}, we compute the Hilbert newforms given by Lemma~\ref{L:Econductors} and we apply the same method as in Section~\ref{bounding} to bound the exponent $p$. 

For the forms at level 
$N_1 = 2^3 w^2$, using the 
auxiliary primes $q =3,17,23,29$, we eliminate all the forms for $p > 7$ except for the two corresponding to the Frey curve evaluated at 
~$(1,-1)$ and~$(1,0)$. 
For $p=7$ it also survives one extra form~$g$ whose field of coefficients is $\Q(\sqrt{2})$.

At level $N_2 = 2^4 w^2$, using the 
auxiliary primes $q = 3,17,23,29,43$, we eliminate all the forms for $p \geq 7$ except for the form corresponding to the Frey curve evaluated at~$(1,1)$\footnote{This last computation takes more than 4 days, with about 3,5 days for the implementation of the space of forms of level~$N_2 = 2^4 w^2$. For the readers interested in checking this result more quickly, we provide online the coefficients of the corresponding newforms for prime ideals of norm~$\leq 200$ (\cite{programs}).}.

Under the assumption $3\mid a+b$, Lemma~\ref{L:Econductors} applies for $p=5$ and  the computations of this proof also, so the last statement follows.
\end{proof}

\begin{remark}
The form $g$ in Proposition~\ref{P:iso2} cannot be eliminated for the exponent $p = 7$, even using `many' auxiliary primes $q \not= 2, 13$. This failure appears to have the following explanation.

Let $\sqrt{2} + 3 \in \fp_7$ and $\sqrt{2} + 4 \in \fp_7'$ be the two primes above~$7$ in $\Q(\sqrt{2})$. By comparing traces of Frobenius mod~$\fp_7'$, we promptly check that $\rhobar_{E,7} \not\cong \rhobar_{g,\fp_7'}$. For the prime~$\fp_7$, the trace of $\rhobar_{g,\fp_7}$ and $\rhobar_{E_{1,-1},7}$ at Frobenius elements for primes~$\fq$ in $\Q(\sqrt{13})$ of norm up to $5000$ are the same, which suggests that $\rhobar_{g,\fp_7} \cong \rhobar_{E_{1,-1},7}$. 

However, to actually have this conclusion, we need to compare traces up to a `Sturm bound'~\cite{BurPac}, which unfortunately turn out to be too large to be computationally feasible.

From the proof of Theorem~\ref{T:4nmidaplusb} below we will see that the possibility of having a congruence with~$g$ as in Proposition~\ref{P:iso2} is the only obstruction for that theorem to hold also for~$p=7$. Therefore, if we could indeed show $\rhobar_{g,\fp_7} \cong \rhobar_{E_{1,-1},7}$, this means $g$ is not problematic as it contains exactly the same mod~$7$ information as  $E_{1,-1}$ which is already part of the conclusions of Proposition~\ref{P:iso2}. Alternatively, if we can show the mod~$7$ congruence of the Frey curve~$E$ with~$g$ is impossible (under the hypothesis contrary to the conclusions of Theorem~\ref{T:4nmidaplusb}), this would also give us Theorem~$\ref{T:4nmidaplusb}$ for $p=7$. Doing this requires new methods. In the paper \cite{BCDDF} of the authors with Lassina Demb\'el\'e, we address this problem which ultimately leads to the resolution of~\eqref{E:1313p} also for $p=7$.

\end{remark}

\begin{proof}[Proof of Theorem~$\ref{T:4nmidaplusb}$] 
Suppose $(a,b,c)$ is a non-trivial
primitive solution to \eqref{E:1313p}
with $p \ge 5$ and $p \not= 7, 13$.
Write $E = E_{a,b}$. From Proposition~\ref{P:iso2}, we know that $\rhobar_{E,p} \cong \rhobar_{Z,p}$, where $Z$ is 
$E_{1,-1}$, $E_{1,0}$ or $E_{1,1}$. 

Let again $\fq_1$ and $\fq_2$ be the primes in $\Q(\sqrt{13})$ dividing~$3$ with $w+1\in\fq_1$. 
Both $\fq_i$ are primes of good reduction for $E$ and $Z$; since $3 \mid d$, we have $3 \mid a+b$ and $a_{\fq_1}(E) = -3$ (see the proof of Proposition~\ref{P:irredE13}).

On the other hand, for $Z= E_{1,0}$ or $Z = E_{1,1}$, we have $a_{\fq_1}(Z) = -1$. Therefore we have $a_{\fq_1}(E) \not\equiv a_{\fq_1}(Z) \pmod{ \fp}$ and we conclude that $\rhobar_{E,p}\cong \rhobar_{E_{1,-1},p}$. 

We now prove (A). Let $K^+$ be the maximal totally real subfield of $\Q(\zeta_{13})$ and $\pi$ denote the prime ideal in $K^+$ above 13. From \cite[Proposition~3.1]{DF2},
 when $13 \nmid a+b$ (or equivalently $13 \nmid c$), the curve 
$Z / K^+$ has good reduction at $\pi$ and $E/K^+$ has bad additive reduction. 
The conclusion follows from version 2 of image of inertia argument.

We now prove (B). Consider the field $M = \Q(\sqrt{13})(x,y)$, where~$(x, y)$ is a $3$-torsion point of $Z = E_{1,-1}$ whose coordinates satisfy:
\begin{align*}
& x^4 + (702 w - 3510)x^2 + (-32994 w + 135486)x + 410670w - 1560546 = 0 \\
& y^2 = x^3 + (351w - 1755)x - \frac{16497w - 67743}{2}.
\end{align*}
The extension $M$ has degree $8$ over $\Q(\sqrt{13})$. Let $\fq'$ be the unique prime of $M$ of ramification index $8$ above the prime $2$ of $\Q(\sqrt{13})$. 

Let~$a'$, $b'$ be coprime integers and consider the base change to~$M$ of~$E = E_{a,b}$ and~$E'=E_{a',b'}$. We have~$v_{\fq'}(\Delta(E)) = v_{\fq'}(\Delta(E')) = 32 \le 36 = 12\cdot 3$.

Write~$a_4$, $a_4'$, $a_6$, $a_6'$ for~$a_4(a,b)$, $a_4(a',b')$, $a_6(a,b)$, $a_6(a',b')$ respectively and recall that~$a_4$, $a_4'$ and~$a_6$, $a_6'$ are homogeneous polynomials of degree~$4$ and~$6$ with coefficients in the integer ring of~$\Q(\sqrt{13})$, respectively.

Suppose that the reduction type of $E'$ is either $II^*$ or $I_0^*$ and we have that $v_{\fq'}(a_4 - a'_4) \geq 4\cdot 3$ and $v_{\fq'}(a_6 - a_6') \geq 6\cdot 3$. By \cite[Lemma 2.1]{BeChDaYa} applied to~$E$ and~$E'$ over~$M$ with~$k = 3$, the reduction types of $E$ and $E'$ at $\fq'$ are the same, and hence the conductor exponents at $\fq'$ of $E$ and $E'$ are the same.  

Therefore, if $(a,b) \equiv (a',b') \pmod{2^3}$ holds, then the assumptions of the previous paragraph are satisfied and then the conductor exponent at $\fq'$ of $E_{a,b}$ is the same as that of $E_{a',b'}$, provided the reduction type of $E_{a',b'}$ is $II^*$ or $I_0^*$. Assuming $4 \nmid a+b$ and using \cite{programs}, it is thus shown that $E_{a,b}/M$ has conductor exponent $\ge 4$ at the prime of $M$ above $2$, whereas $Z/M$  has conductor exponent $2$ at the prime of $M$ above $2$. The conclusion follows from version~3 of the image of inertia argument. We note that the full $3$-division field of $Z$ has degree $48$ over $\Q(\sqrt{13})$, whereas our choice of $M$ has degree $8$ over $\Q(\sqrt{13})$, making the computation faster.
\end{proof}

\subsection{The modular method over the real cubic subfield of $\Q(\zeta_{13})$}
\label{S:irreducibilityF}

In \cite{F}, several Frey elliptic curves are attached to equation~\eqref{E:rrp}.
In particular, for $r=13$ one of them is $E_{a,b}$ from the previous section; in this 
section we will use another Frey elliptic curve adapted from a construction in {\it loc.\ cit.} defined over a cubic field.

Let $K^+$ be the maximal (degree 6) totally real subfield of $\Q(\zeta_{13})$
and write $K$ for its cubic subfield. Write $\zeta = \zeta_{13}$ and define
\[
 A_{x,y} = \alpha(x+y)^2, \quad
 B_{x,y} = \beta (x^2 + (\zeta +\zeta^{-1})xy+y^2), \quad 
 C_{x,y} = \gamma (x^2+(\zeta^8+\zeta^{-8})xy+y^2), 
\]
where
\[
 \alpha = \zeta^8 + \zeta^{-8} - \zeta - \zeta^{-1}, \quad
 \beta  = 2-\zeta^8 - \zeta^{-8}, \quad 
 \gamma = \zeta +\zeta^{-1}-2
\]
all have norm $13^2$. We note that $A_{x,y}$, $B_{x,y}$, $C_{x,y}$ are polynomials with coefficients in~$K^+$ satisfying
$A_{x,y} + B_{x,y} + C_{x,y} = 0$.

Let $a,b$ be coprime integers such that~$a+b\not=0$. We consider the Frey elliptic curve given by the following short Weierstrass equation
\[
 F_{a,b} \; : \; y^2 = x^3+a_4'(a,b)x+a_6'(a,b)
\]
where 
\begin{eqnarray*}
 a_4'(a,b) & =  & 3^3\cdot 13^2\left(A_{a,b}B_{a,b}+A_{a,b}C_{a,b}+B_{a,b}C_{a,b}\right)  \\
 a_6'(a,b) & = &  -3^3\cdot 13^3\left(2A_{a,b}^3+3A_{a,b}^2B_{a,b}-3A_{a,b}B_{a,b}^2-2B_{a,b}^3\right).
\end{eqnarray*}
This curve is (up to a rational isomorphism) the quadratic twist by~$13$ of the curve defined by equation~(13) with $(k_1,k_2)=(1,8)$ in~\cite{F}.

We then verify that $F_{a,b}$ is defined over~$K$. Its standard invariants are  given by the following identities~:
\begin{eqnarray*}
c_4(F_{a,b}) & = & -2^4\cdot3\cdot a_4'(a,b) = -2^4\cdot3^4\cdot 13^2\left(A_{a,b}B_{a,b}+A_{a,b}C_{a,b}+B_{a,b}C_{a,b}\right)\\
\Delta(F_{a,b}) & = & 2^{4} \cdot 3^{12} \cdot 13^6\left(A_{a,b}B_{a,b}C_{a,b}\right)^2 .
\end{eqnarray*}

We now determine the conductor of~$F_{a,b}$. For simplicity, let us write $F=F_{a,b}$ and~$N_F$ for its conductor.

\begin{lemma}\label{lem:bad_reduction_1313p_FI}
Let~$\fq$ be a prime ideal in~$K$ of residual characteristic~$\ell\neq2, 3,13$. If $F$ has bad reduction at~$\fq$ then~$\ell\mid a^{13}+b^{13}$. Furthermore, if $\ell\not\equiv1\pmod{13}$, then $F$ has bad reduction at~$\fq$ if and only if~$\ell\mid a+b$. Moreover, if $F$ has bad reduction at~$\fq$, we have
$$
\vv_{\fq}(c_4(F))=0\quad\text{and}\quad \vv_{\fq}(\Delta(F))=\delta\vv_\ell(a^{13}+b^{13}),
$$
where~$\delta=2$ or~$4$ according to whether~$\ell$ divides~$\phi_{13}(a,b)$ or~$a+b$ respectively; in particular, $F$ has bad multiplicative reduction at~$\fq$ and hence~$\vv_{\fq}(N_F)=1$.
\end{lemma}
\begin{proof}
Recall that $A=A_{a,b}$, $B=B_{a,b}$, $C=C_{a,b}$ are coprime outside~$13$ as elements of~$\Q(\zeta)$.  Moreover $(ABC)^2$ divides $13(a+b)^2(a^{13}+b^{13})^2$ and the quotient is coprime to~$(ABC)^2$ away from~$13$ (see Section~$2$ of~\cite{DF2}). In particular, if $F$ has bad reduction at~$\fq$ then~$\ell$ divides~$ a^{13}+b^{13}$ and $\fq\nmid c_4(F)$. If $\ell\not\equiv1\pmod{13}$, the equivalence holds since primes dividing~$a^{13}+b^{13}$ not congruent to~$1$ modulo~$13$ automatically divide~$a+b$ and hence~$A$.  Moreover, in that case, we have 
\[
\vv_{\fq}(\Delta(F))=\vv_{\fq}\left((ABC)^2\right)=4\vv_{\ell}(a+b)+2\vv_{\ell}(\phi_{13}(a,b)).
\] 
The result then follows from the fact that $a+b$ and~$\phi_{13}(a,b)$ are coprime outside~$13$.
\end{proof}

We now determine the valuation of~$N_F$ at the unique prime ideals above $2$, $3$ and~$13$. The two former prime numbers are inert in~$K$ and we simply write~$2$ and~$3$ for the unique primes above them in~$K$. We denote by~$\fq_{13}$ the prime ideal above~$13$ in~$K$. 
\begin{lemma}\label{lem:bad_reduction_1313p_FII}
We have the following valuations~:
\begin{eqnarray*}
\vv_{\fq_{13}}(N_F)  & = &  \left\{
\begin{array}{ll}
1 & \text{if }13\mid a+b \\
2 & \text{if }13\nmid a+b \\
\end{array}
\right.; \\
\vv_3(N_F)  & = & \left\{
\begin{array}{ll}
0 & \text{if }3\nmid a+b \\
1 & \text{if }3\mid a+b
\end{array}
\right.; \\
\vv_2(N_F) &  =  & \left\{
\begin{array}{ll}
0 & \text{if }\vv_2(a+b)=2\\
1 & \text{if }\vv_2(a+b)\ge 3\\
3 & \text{if }ab\equiv 0\pmod{4}\\
4 & \text{if  $\vv_2(a+b)=1$ or $ab\equiv 2\pmod{4}$}\\
\end{array}
\right..
\end{eqnarray*}
\end{lemma}
\begin{proof} For simplicity, write $A=A_{a,b}$, $B=B_{a,b}$ and~$C=C_{a,b}$. Let us denote by~$\pi_{13}$ the unique prime ideal in~$\Q(\zeta)$ above~$13$. We first compute the valuation at~$\pi_{13}$ of~$c_4(F)$ and~$\Delta(F)$. Using the equalities $AB=\alpha\beta(a+b)^2\left((a+b)^2+\gamma ab\right)$ and $AC=\alpha\gamma(a+b)^2\left((a+b)^2-\beta ab\right)$ we obtain
\begin{equation*}
\vv_{\pi_{13}}(AC)=\vv_{\pi_{13}}(AB) = 24\vv_{13}(a+b)+\left\{\begin{array}{ll}
					6 & \text{if $13\mid a+b$} \\
					4 & \text{if $13\nmid a+b$}
					\end{array}
					\right..
\end{equation*}
Similarly, using $BC=\beta\gamma\left((a+b)^2+\gamma ab\right)\left((a+b)^2-\beta ab\right)$, we have
\begin{equation*}
\vv_{\pi_{13}}(BC)= \left\{\begin{array}{ll}
					8 & \text{if $13\mid a+b$} \\
					4 & \text{if $13\nmid a+b$}
					\end{array}
					\right..
\end{equation*}
Therefore it follows that we have
\begin{equation*}
\left(\vv_{\fq_{13}}(c_4(F)),\vv_{\fq_{13}}(\Delta(F))\right)=
					\left\{\begin{array}{ll}
					\left(8,23 +12 \vv_{13}(a+b)\right) & \text{if $13\mid a+b$} \\
					\left(\ge 7,21\right) & \text{if $13\nmid a+b$}
					\end{array}
					\right..
\end{equation*}
In particular, the defining model of~$F$ is not minimal at~$\fq_{13}$ (\cite[Tableau~I]{papado}). After a change of variables, we obtain that if $13$ divides~$a+b$, then $F$ has bad multiplicative reduction of type~$I_{\nu}$ with $\nu=-1+12\vv_{13}(a+b)$. Therefore we have~$\vv_{\fq_{13}}(N_F)=1$. Otherwise, if $13$ divides~$a+b$, then $F$ has bad additive reduction at~$\fq_{13}$ and~$\vv_{\fq_{13}}(N_F)=2$.

We now deal with the prime ideal generated by~$3$. Neither~$B$ nor $C$ is divisible by~$3$. In particular, if $3$ does not divide~$a+b$, then $(\vv_3(c_4(F)),\vv_3(\Delta(F)))=(\ge 4,12)$ and the defining model of~$F$ is not minimal at~$3$ (\cite[Tableau~I]{papado}). After a change of variables, we obtain that $F$ has good reduction at~$3$; hence $\vv_3(N_F)=0$.  Otherwise, if $3$ divides~$a+b$, then $(\vv_3(c_4(F)),\vv_3(\Delta(F)))=(4,12+4\vv_3(a+b))$. Therefore, according to \emph{loc. cit.}, after a change of variables, we obtain that $F$ has bad multiplicative reduction of type~$I_{\nu}$ with $\nu=4\vv_3(a+b)$. Therefore we have~$\vv_3(N_F)=1$.

We finally compute the valuation at~$2$ of the conductor of~$F$. Neither~$B$ nor~$C$ is divisible by~$2$. Therefore, we have
\[
\vv_2(c_4(F))=4+\vv_2(AB+AC+BC)\quad\text{and}\quad \vv_2(\Delta(F))=4+4\vv_2(a+b).
\]
In particular, if $\vv_2(a+b)\ge 3$, then after a change of variables, we found that $F$ has bad multiplicative reduction at~$2$ of type~$I_\nu$ with $\nu=-8+4\vv_2(a+b)$; hence $\vv_2(N_F)=1$. Similarly, if $\vv_2(a+b)=2$, then $F$ has good reduction at~$2$ and $\vv_2(N_F)=0$.

It remains to deal with the case~$\vv_2(a+b)\le 1$. Assume first that $2$ does not divide~$a+b$. Then we have $ab\equiv 0\pmod{2}$ and $(\vv_2(c_4(F)),\vv_2(\Delta(F)))=(4,4)$. Therefore,  by~\cite[Tableau~IV]{papado}, we are in Case~$3$, $4$ or~$5$ of Tate's classification and $\vv_2(N_F)=4$, $3$ or~$2$ respectively. We have
\[
a_4'(a,b)\equiv -(\alpha\beta+\alpha\gamma+\beta\gamma+\alpha\beta\gamma ab)\pmod{4}
\]
and
\[
a_6'(a,b)\equiv 2\alpha^3+3\alpha^2\beta-3\alpha\beta^2-2\beta^3+3\alpha^2\beta\gamma ab\pmod{4}.
\]
In particular, we have $a_4'(a,b)\equiv \alpha\beta+\alpha\gamma+\beta\gamma\equiv r^2\pmod{2}$ and $a_6'(a,b)\equiv \alpha\beta\gamma\equiv t^2\pmod{2}$ with $r=t=\zeta+\zeta^{-1}+\zeta^2+\zeta^{-2}+\zeta^3+\zeta^{-3}+\zeta^5+\zeta^{-5}$. According to~\cite[Prop.~1]{papado}, one then verifies using the above congruences of~$a_4'(a,b)$ and~$a_6'(a,b)$ modulo~$4$ that we are in a Case~$\ge 4$ of Tate's classification if and only if $ab\equiv 0\pmod{4}$.  In that case, the congruence class (in the notation of \emph{loc. cit.}) of $b_8+3rb_6+3r^2b_4+r^3b_2+3r^4=-a_4'(a,b)^2+12ra_6'(a,b)+6r^2a_4'(a,b)+3r^4$ modulo~$2^3$ is independent of~$a$, $b$ such that $ab\equiv 0\pmod{4}$ (since then both $a_4'(a,b)\pmod{4}$ and~$a_6'(a,b)\pmod{2}$ only depend on~$\alpha$, $\beta$, and $\gamma$). One then verifies that it has valuation~$2$. According to~\cite[Prop.~1]{papado} we are in Case~$4$ of Tate's classification; hence $\vv_2(N_F)=3$.

Assume now that $\vv_2(a+b)=1$. Then we have $(\vv_2(c_4(F)),\vv_2(\Delta(F)))=(4,8)$ and  by~\cite[Tableau~IV]{papado}, we are in Case~$6$, $7$ or~$8$ of Tate's classification. We have
\[
a_4'(a,b)\equiv \beta\gamma(4\alpha-3\beta\gamma)\pmod{8}\quad\text{and}\quad a_6'(a,b)\equiv 2(\beta\gamma)^3\pmod{4}.
\]
Since the congruence class (in the notation of \emph{loc. cit.}) of $b_8+3rb_6+3r^2b_4+r^3b_2+3r^4=-a_4'(a,b)^2+12ra_6'(a,b)+6r^2a_4'(a,b)+3r^4$ modulo~$2^4$ only depends on  $a_4'(a,b)\pmod{8}$, $a_6'(a,b)\pmod{4}$, it is independent of~$a,b$ such that~$\vv_2(a+b)=1$. In particular, we can take $(a,b)=(1,1)$. Using {\tt Magma} we check that the elliptic curve~$F_{1,1}$ has conductor exponent~$4$ at~$2$. This means that for this specific curve we are in Case~$6$ of Tate's classification. In particular, the congruence equation in~\cite[Prop.~3(a)]{papado} has no solution for $(a,b)=(1,1)$ and hence for all $a,b$ with $\vv_2(a+b)=1$. It follows that $\vv_2(N_F)=4$ when $\vv_2(a+b)=1$.
\end{proof}

Write $E = E_{a,b}$ and $F = F_{a,b}$.
The following illustrates a fundamental difference between the Frey elliptic curves
$E$ and $F$. Note that irreducibility of $\rhobar_{E,p}$ followed by 
an application of \cite[Theorem~3]{FS} which makes crucial use of the presence of
explicit primes of good reduction of~$E$. This was guaranteed by the fact that all the primes 
not dividing $2 \cdot 13$ of bad reduction of $E$  
must have residual characteristic congruent
to 1 mod~$13$ (see Lemma~\ref{lem:bad_reduction_1313p_EI}). This is no longer the case for~$F$ 
due to the factor~$a+b$ in~$\Delta(F)$. Therefore, we can only apply 
\cite[Theorem~2]{FS} which guarantees that $\rhobar_{F,p}$ is irreducible
when $p > (1 + 3^{18})^2$. This bound is insufficient for our purposes.

We shall establish here a much better irreducibility result, dealing first with the case~$p=5$ in full generality.
\begin{lemma}	
The representation~$\rhobar_{F,5}$ is irreducible.
\label{L:irredFmod5}
\end{lemma}

\begin{proof}
We proceed using explicit equations as in~\cite[Theorem~7]{Dahmen}.
Let $j_F$ denote the $j$-invariant of $F$. Then
\begin{align*}
  j_F - 1728 & = \eta G(a,b)^2/H(a,b)^2, \\
   & = 13 (\nu G(a,b)/H(a,b))^2,
\end{align*}
where $G, H$ are degree-$6$ homogeneous monic polynomials in two variables with coefficients in~$K$, and $\eta, \nu \in K$ \cite{programs}. If $\rhobar_{F,5}$ is reducible, that is, $F$ has a $5$-isogeny over $K$, then we must have that
\begin{equation*}
  j_F - 1728 = \frac{(t^2 + 4 s t - s^2)^2 (t^2 + 22 st + 125 s^2)}{s^5 t},
\end{equation*}
for some $(s,t) \in K\setminus\{0\}$, following the argument in~\cite{Dahmen}. Thus, we obtain a $K$-rational point~$(u,v)$ on the elliptic curve
\begin{equation*}
D : 13 Y^2 = (X^2+22X+125)X,
\end{equation*}
where
\begin{equation*}
  u = t/s , \qquad
  v  = \nu \frac{G(a,b)}{H(a,b)} \frac{u}{u^2 + 4 u - 1}.
\end{equation*}
The elliptic curve $D$ has rank $1$ over $K$ and over $\Q$ and $D_{\text{tors}}(K) = D_{\text{tors}}(\Q) \cong \Z/2\Z$ by~\cite{programs}. We claim that this implies $D(K) = D(\Q)$, i.e.\ every $K$-rational point of $D$ is in fact $\Q$-rational. Thus,
\begin{align*}
  \nu \frac{G(a,b)}{H(a,b)} \in \Q,
\end{align*}
which in turn implies that $j_F -1728 \in \mathbb{P}^1(\Q)$. Put $N=\eta G(a,b)^2$, $M=H(a,b)^2$ and let~$\sigma$ be a non-trivial automorphism of~$K$. Write $R = N\sigma(M)\sigma^2(M)=A_0(a,b)+A_1(a,b)z+A_2(a,b)z^2$ where $K=\Q(z)$ and $A_0,A_1,A_2$ are degree~$36$ homogeneous polynomials in two variables with coefficients in~\(\Q\). Since $j_F -1728 \in \mathbb{P}^1(\Q)$, then $R$ is rational and this implies that $A_1(x,1)$ and $A_2(x,1)$ must have a common root. It can be verified with {\tt Magma} \cite{programs} that this is not the case for $a/b \in \mathbb{P}^1(\Q)$, except for $a/b = -1$, as desired.

To complete the proof we prove the claim.
Recall that $D$ has rank $1$ over both $\Q$ and $K$
and $D_{\text{tors}}(K) = D_{tors}(\Q) \cong \Z/2\Z$.
Therefore, $D(\Q)$ is of finite index in $D(K)$ and
there is an integer $N > 1$ such that $[N] ( D(K) )\subset D(\Q)$.
Now let $P \in D(K)$. Thus $Q = N \cdot P \in D(\Q)$.
Let $\calF_Q$ denote the inverse image of $Q$ by $[N]$ and $\sigma \in G_\Q$. Thus $P$ and $\sigma(P)$ are in $\calF_Q$ because $D$ has a rational model. Further, since $K/\Q$ is Galois, we have $\sigma(P) \in D(K)$.
Therefore, $N \cdot (\sigma(P) - P) = Q - Q = 0$ so $\sigma(P)  = P + T$ with $T \in D(K)_{\tors} = D(\Q)_{\tors}$. Furthermore,
\[
\sigma^2 (P) = \sigma (P) + \sigma (T) = \sigma(P) + T = P + 2T = P
\]
where we used that the addition law is defined over $\Q$ because $D$ has a rational model. Finally, observe that $\sigma^2$ restricted to~$K$ is either the identity or a generator of $\Gal(K/\Q)$. We conclude that $P$ is invariant by the action of $\Gal(K/\Q)$, hence $P \in D(\Q)$.
\end{proof}
\begin{remark}
The following example shows that in the previous proof of Lemma~\ref{L:irredFmod5} it is essential 
to be working with $j$-invariants arising from the Frey elliptic curve $F$.
Consider the elliptic curve over~$\Q$ defined by
\[
 y^2 + (1+a)xy +ay = x^3 + ax^2, \quad a = \frac{-10933}{144}
\]
which has $10$-torsion over $\Q$ and acquires 
full $2$-torsion over $\Q(\sqrt{13})$. In particular, it also 
has $10$-torsion over $K$ and a $C_2 \times C_{10}$ torsion group over $K^+$.
\label{R:example}
\end{remark}

\begin{theorem} Assume~$p \geq 7$ and $p \neq 13$. If either $13\mid a+b$, or $13\nmid a+b$ and~$p\not=17,37$, then $\rhobar_{F,p}$ is irreducible.
\label{T:irreducibilityF}
\end{theorem}

\begin{proof} 
Suppose $\rhobar_{F,p}$ is reducible, that is,
\[ \rhobar_{F,p} \sim \begin{pmatrix} \theta & \star\\ 0 & \theta' \end{pmatrix} 
\quad \text{with} \quad \theta, \theta' : G_K \rightarrow \F_p^* 
\quad \text{satisfying} \quad \theta \theta' = \chi_p.\]

We note that $K = \Q(z)$, where $z^3 + z^2 - 4z + 1=0$. 
According to the notation of \cite[Theorem~1]{FS}
we set $\epsilon_1 = z$ and $\epsilon_2 = 1-z$,
observe that the unit group of $K$ is 
generated by $\{-1, \epsilon_1, \epsilon_2\}$ and compute $B = 5^3 \cdot 13$.
Thus from the first paragraph of the proof of \cite[Theorem~1]{FS} we 
conclude that for $p=11$ and $p \geq 17$ exactly one of $\theta$, $\theta'$ 
ramifies at $p$. Since $7$ is inert in $K$ and $F$ is semistable at 7, it follows from 
\cite[Lemma 1]{Kraus3} also that only one of $\theta$, $\theta'$ ramifies at $p=7$.

The characters $\theta$ and $\theta'$ ramify only at $p$ and additive primes of $F$;
the latter are $\fq_{13}$ and~$2$ when $13\nmid a+b$ and~$4\nmid a+b$ respectively (see Lemma~\ref{lem:bad_reduction_1313p_FII}). Furthermore, at an additive prime~$\fq$ both $\theta$, $\theta'$ have conductor 
exponent equal to $\vv_\fq(N_F)/2$; in particular, $\vv_2(N_F) \neq 3$.

Replacing $F$ by a $p$-isogenous curve we can assume $\theta$ is unramified at $p$.
Therefore, the possible conductors for~$\theta$ are $2^s\fq_{13}^t$ with $s\in\{0,2\}$ and~$t\in\{0,1\}$.
Let $\infty_1$, $\infty_2$ and $\infty_3$ be the real places of $K$. The field~$K$ has narrow class number~$1$ and the Ray class groups for the modulus $2^2 \infty_1 \infty_2 \infty_3$, $\fq_{13} \infty_1 \infty_2 \infty_3$
and $2^2 \fq_{13} \infty_1 \infty_2 \infty_3$ are isomorphic to
\[
\Z/2\Z \oplus \Z/2\Z \oplus \Z/2\Z,\quad \Z/4\Z \quad \text{ and } \quad \Z/4\Z \oplus \Z/2\Z \oplus \Z/2\Z \oplus \Z/2\Z,
\]
respectively; hence $\theta$ has order $n=1,2$ or 4. Moreover the case~$n=4$ only occurs when~$\theta$ ramifies at~$\fq_{13}$. In particular, if $13\mid a+b$, then $F$ is semistable at~$\fq_{13}$ and we have~$n=1$ or~$2$. 

Suppose $n = 1,2$. Thus either $F$ or a quadratic twist 
$F'$ of $F$ has a $p$-torsion point defined over $K$. 
Note that $F$ has full 2-torsion over $K^+$ which is a quadratic extension of $K$, hence it 
has at least one $2$-torsion point over $K$ (namely the point with $x$-coordinate $-3\cdot 13(A_{a,b}+2B_{a,b})$). Thus, the quadratic twist $F'$ also has
a $2$-torsion point over $K$ and we conclude that the $K$-torsion subgroup
of $F$ or $F'$ has order divisible by $2p$ with $p \geq 7$. 
From \cite[Theorem 5]{BruNaj}, we see that this is impossible. 

In particular, this proves the result for all primes $p \equiv 3 \pmod{4}$, because $n = 4$ does not
divide the order of $\F_p^*$.

Suppose $n=4$.  Since $K^+$ is the field fixed  by $\theta^2$ (note $\theta^2$ has conductor $\fq_{13}$) and thus $\theta$ has order 2 over $K^+$. After a quadratic twist, now over $K^+$, 
we conclude that $F$ has a $p$-torsion point defined over $K^+$. 
From \cite{smallTorsion} we see this is possible only for $p \leq 19$ and $p = 37$.
We conclude that $\rhobar_{F,p}$ is irreducible for all $p \geq 7$ such that
$p \neq 13, 17, 37$ (after discarding the primes $p \equiv 3 \pmod 4$). 
\end{proof}

From Lemma~\ref{lem:bad_reduction_1313p_FII} we know that $F_{a,b}$ is semistable at all primes dividing~$3$ in $K$. Thus, from \cite[Theorem~6.3]{F}, it follows that $F_{a,b}$ is modular. 

We now wish to use the Frey elliptic curve~$F_{a,b}$ to solve our Fermat equations. Suppose that there exists an integer~$c$ such that $(a,b,c)$ is a non-trivial primitive solution to \eqref{E:1313p} with $p\geq 5$. Write again~$F = F_{a,b}$.

From the conductor computations coming from Lemmas~\ref{lem:bad_reduction_1313p_FI} and~\ref{lem:bad_reduction_1313p_FII}, irreducibility results from Lemma~\ref{L:irredFmod5} and Theorem~\ref{T:irreducibilityF} and level lowering again, we obtain the following lemma.
\begin{lemma}
Assume~$p\ge5$, $p\not=13$. If~$13\nmid a+b$, assume further~$p\not=17, 37$. 
Then, there exist a prime~$\fp$ in~$\Qbar$ above~$p$ such that
\begin{equation}
  \rhobar_{F,p} \cong \rhobar_{f,\fp}, 
  \label{E:Fiso}
\end{equation}
where $f$ is a Hilbert newform over $K$ of parallel weight $2$, trivial character 
and level 
\[
N_f = 2^{\alpha_2}\cdot 3^{\alpha_3}\cdot \fq_{13}^{\alpha_{13}}.
\]
Here, $\alpha_2\in\{0,1,3,4\}$, $\alpha_3\in\{0,1\}$, $\alpha_{13}\in\{1,2\}$ is the valuation (computed in Lemma~\ref{lem:bad_reduction_1313p_FII}) at~$2$, $3$ and~$\fq_{13}$ of~$N_F$ respectively.

\label{L:Fconductors}
\end{lemma}

We now comment on the sizes of the spaces occuring
in Lemma~\ref{L:Fconductors}. 
With {\tt Magma}, we compute respectively the dimensions 
of the cuspidal and its new subspace at each level of the form $2^s\cdot 3\cdot\fq_{13}^t$ with $s\in\{0,1,3,4\}$ and~$t\in\{1,2\}$. For~$t=1$ and $t=2$, we obtain:
\[
\begin{array}{lcl}
 s=0\colon\quad 33, 27;    &\quad \text{and} \quad &  s=0\colon\quad 425, 334 ; \\
 s=1\colon\quad  295, 181 ;  & & s=1\colon\quad 3823, 2353 ; \\
 s=3\colon\quad  18817, 11466;  & & s=3\colon\quad 244609, 148101 ; \\ 
 s=4\colon\quad 150929, 91728  ; & & s=4\colon\quad 1956865, 1184820.\\
\end{array}
\]
We see that for $s=0$ and $s=1$, the computations of the newforms are within reach of current
implementations (indeed, we have already computed a larger space when studying the case of 
$r=5$), but for $s=3$ and $s=4$, the dimensions are totally out of reach. 
Using the multi-Frey technique, we are able to prove Theorem~\ref{T:main13} by
computing only in the case $(s,t)=(1,1)$ (that is in level~$2\cdot 3\cdot \fq_{13}$).
 
\subsection{Proof of Theorem~$\ref{T:main13}$}\label{S:Thm2}
The case $p = 2$ follows from \cite[Theorem 1.1]{BenSki} and the case $p = 3$ follows from \cite[Theorem 1.5]{BenVatYaz}. For exponent $p=13$ the result follows from Theorem~2 in \cite[Section~4.3]{Serre87}.

Suppose $(a,b,c)$ is a non-trivial primitive solution to \eqref{E:1313p} with $p \geq 5$,  $p \neq 7, 13$.
From Theorem~\ref{T:4nmidaplusb} we can assume that $4 \mid a+b$ and $ 13 \mid a+b$. Moreover, we have $\vv_2(a+b)=\vv_2(3c^p)\ge3$. Write $F = F_{a,b}$. Thanks to our assumptions, Lemma~\ref{L:Fconductors} applies with no further restrictions. In particular, we have that $\rhobar_{F,p} \cong \rhobar_{f,\fp}$, 
where $f$ is a Hilbert newform over~$K$ of parallel weight~$2$, trivial character and level $N_f = 2\cdot 3\cdot \fq_{13}$.
(Note that the multi-Frey technique is implicit in this step 
because the proof of Theorem~\ref{T:4nmidaplusb} uses the 
Frey elliptic curve $E$.)

The dimension of the new cuspidal subspace is~$181$. 
Using \cite{programs}, we compute all the $15$~newforms in this space and  bound the exponent using the primes in~$K$ above rational primes $q=5,7,11,17,31$ as usual using the norm of the difference between traces. This suffices to eliminate all but $4, 2$ forms corresponding 
to the exponents $p = 5, 11$, respectively.


For the remaining forms, we use the following refined elimination technique~\cite{programs}. For each form, choosing a $q\not=2,3,13$ and $q\not\equiv1\pmod{13}$, we obtain that if $\rhobar_{F,p} \cong \rhobar_{f,\fp}$ for some prime $\fp \mid p$ in the field of coefficients of~$f$ (by Lemma \ref{lem:bad_reduction_1313p_FI}):
\begin{enumerate}[(i)]
\item\label{item:good} either $q\nmid a+b$ and, for all~$\fq$ above~$q$, we have $a_{\fq}(f)\equiv a_{\fq}(F_{a,b})\pmod{\fp}$;
\item\label{item:bad} or $q\mid a+b$ and, for all~$\fq$ above~$q$, we have $a_{\fq}(f)\equiv \pm(N(\fq)+1)\pmod{\fp}$.
\end{enumerate}
By computing~$a_{\fq}(F_{x,y})$ for each~$\fq\mid q$ and all~$x,y\in\{0,\dots,q-1\}$ not both zero, we eliminate each form by checking that neither of the above congruences hold for that form for all choices of $\fp \mid p$ in~$\Q_f$, where~$\Q_f$ denotes the field of coefficients of~$f$. Note this is computationally practical as long as we can factor~$p$ in~$\Q_f$. Using the auxiliary prime $q=5$ we deal with the two remaining forms for $p=11$ and with the auxiliary primes $q=31,47$ we deal with the remaining forms for $p=5$ except for the form~$f_{13}$ (i.e. the form in position 13 as given by {\tt Magma});  see~\cite{programs}.

To complete the proof we have to show that $\rhobar_{F,5} \cong \rhobar_{f_{13},\fp}$ is not possible for all $\fp \mid 5$ in~$\Q_{f_{13}}$, and we will do this using the multi-Frey technique one last time. Indeed, using the auxiliary primes $q=11,31$ and applying the same refined elimination as above we discard $\rhobar_{F,5} \cong \rhobar_{f_{13},\fp}$ except for one prime $\frak{p} = \frak{P} \mid 5$ in $\Q_{f_{13}}$ whose residue field is $\F_5$ (there are two other primes with residue field~$\F_5$ but do not cause problems). Next we use the auxiliary prime $q=19$. Note that 19 is inert in~$K$ and let $\fq_{19} = 19\calO_K$, where~$\calO_K$ denotes the integer ring of~$K$. We check that
$$
a_{\fq_{19}}(f_{13})^2 -(N(\fq_{19})+1)^2 \not\equiv 0 \pmod{\mathfrak{P}},
$$
which means that $f_{13}$ does not satisfy the level raising condition mod~$\mathfrak{P}$  at $\fq_{19}$. Therefore, if $\rhobar_{F,5} \cong \rhobar_{f_{13},\mathfrak{P}}$ holds then $F$ has good reduction at~$\fq_{19}$. Under this assumption, by comparing traces of Frobenius at~$\fq_{19}$, we conclude that $(a,b) \pmod{19}$ belongs to an explicit restricted list $\mathcal{L}$ of pairs mod~$19$.

We now go back to the Frey curve $E = E_{a,b}$ over $\Q(\sqrt{13})$ from Section \ref{S:overQ13}. From the last sentence of Proposition \ref{prop:irredE} we know that we have a mod 5 congruence between $E$ and one of $E_{1,0}$, $E_{1,1}$ or $E_{1,-1}$; from the second and third paragraphs in the proof of~Theorem~\ref{T:main13} we conclude that $\rhobar_{E,5} \cong \rhobar_{E_{1,-1},5}$. Note that 19 is inert in $\Q(\sqrt{13})$. Finally, using the fact that we know that~$(a,b) \pmod{19}$ belongs to $\calL$, by comparing traces of Frobenius at $q=19$ in the previous isomorphism, we obtain the desired contradiction.


\end{document}